\documentclass[11pt]{article}
\usepackage[letterpaper]{geometry}
\usepackage[parfill]{parskip}
\usepackage{amsmath,amsthm,amssymb,bbm,bm}
\usepackage{mathtools}
\usepackage{cases}
\usepackage{graphicx}
\usepackage{tabularx}
\usepackage{microtype}
\usepackage{enumitem}

\usepackage{authblk}

\usepackage{url}
\usepackage[colorlinks,citecolor=blue,urlcolor=blue,linkcolor=blue,linktocpage=true]{hyperref}
\usepackage{cleveref}
\crefformat{equation}{(#2#1#3)}
\crefrangeformat{equation}{(#3#1#4) to~(#5#2#6)}
\crefname{equation}{}{}
\Crefname{equation}{}{}

\usepackage[authoryear]{natbib}
\usepackage{multirow}

\newtheoremstyle{mythmstyle}
  {8 pt} 
  {3 pt} 
  {} 
  {} 
  {\bfseries} 
  {.} 
  {.5em} 
  {} 

\theoremstyle{plain}

\makeatletter
\def\thm@space@setup{%
  \thm@preskip=6pt plus 1pt minus 1pt
  \thm@postskip=\thm@preskip 
}
\makeatother

\newtheorem{theorem}{Theorem}[section]
\newtheorem{lemma}[theorem]{Lemma}
\newtheorem{corollary}[theorem]{Corollary}
\newtheorem{proposition}[theorem]{Proposition}
\newtheorem{remark}{Remark}

\newtheorem*{example*}{Example}
\newtheorem{definition}{Definition}
\newtheorem*{definition*}{Definition}
\newtheorem*{remark*}{Remark}
\crefname{definition}{\textbf{definition}}{definitions}
\Crefname{definition}{Definition}{Definitions}
\crefname{assumption}{\textbf{assumption}}{assumptions}
\Crefname{assumption}{Assumption}{Assumptions}

\begin{document}
\allowdisplaybreaks
\title{Convergence and Concentration of Empirical Measures under Wasserstein Distance in Unbounded Functional Spaces}

 \author[1]{Jing Lei}
\affil[1]{Carnegie Mellon University}

\maketitle

\begin{abstract}
We provide upper bounds of the expected Wasserstein distance between a probability measure and its empirical version, generalizing recent results for finite dimensional Euclidean spaces and bounded functional spaces. Such a generalization can cover Euclidean spaces with large dimensionality, with the optimal dependence on the dimensionality. Our method also covers the important case of Gaussian processes in separable Hilbert spaces, with rate-optimal upper bounds for functional data distributions whose coordinates decay geometrically or polynomially. Moreover, our bounds of the expected value can be combined with mean-concentration results to yield improved exponential tail probability bounds for the Wasserstein error of empirical measures under Bernstein-type or log Sobolev-type conditions.
%
\end{abstract}

\section{Introduction}
Let $\mu$, $\nu$ be two probability measures on a Banach space $\mathcal X$, the \emph{Wasserstein distance} of order $p$ ($p\in [1,\infty)$) between $\mu$, $\nu$ is
\begin{align*}
W_p(\mu,\nu) = \inf_{\xi\in \mathcal C(\mu,\nu)}\left[\mathbb E_{(X,Y)\sim \xi}\|X-Y\|^p\right]^{1/p}\,,
\end{align*}
where $\mathcal C(\mu,\nu)$ denotes the collection of all distributions on $\mathcal X^2$ with marginal distributions being $\mu$ and $\nu$.

Wasserstein distances have a clear intuitive meaning:  What is the minimum cost if we want to obtain $\nu$ by transporting the probability mass in $\mu$? Here the cost is defined as the product of probability mass moved and the distance moved raised to the $p$th power. Therefore, the Wasserstein distance is also called ``optimal transport distance'' or ``earth mover's distance''. The problem of optimal transport can be traced back to \cite{Monge1781} and \cite{Kantorovich42}.  

Since the introduction in \cite{Vaserstein69}, the Wasserstein distances have become an important tool in computer vision and statistical machine learning.   In addition to the connection with optimal transport, Wasserstein distances have some desirable features.  For example, they can be meaningfully defined for any two distributions without any requirement on the existence of density or absolute continuity.  Other measurements, such as the Kullback-Leibler divergence, have more stringent requirements on $\mu$ and $\nu$. See \cite{SommerfeldM18} for a more thorough historical review of Wasserstein distances and their applications in statistics, and \cite{Villani08} for further details about Wasserstein distances and optimal transport in a broader context.

In statistics and machine learning, we often do not have access to $\mu$ but only its empirical version $\hat\mu$, which puts $1/n$ mass at each one of $n$ independent samples from $\mu$.  A fundamental problem is to understand $W_p(\hat\mu,\mu)$.  When $p=1$ the Kantorovich-Rubinstein duality \citep{KantorovichR58} implies that $W_1(\hat\mu,\mu)$ is equivalent to the supremum of the empirical process indexed by Lipschitz-$1$ functions. As a consequence, \cite{Dudley69} provides sharp lower and upper bounds of $\mathbb E W_1(\hat\mu,\mu)$ for $\mu$ supported on a bounded finite dimensional set. \cite{Talagrand94transportation} studies the case when $\mu$ is the uniform distribution on a $d$-dimensional unit cube.  For general distributions, \cite{Boissard14,DereichSS13,FournierG15} establish sharp upper bounds of $\mathbb E W_p(\hat\mu,\mu)$ in finite dimensional Euclidean spaces.  Recently, \cite{WeedB17,SinghP18} provide similar results for general distributions on a bounded metric space $\mathcal X$.

In this paper we study upper bounds of $\mathbb E W_p(\hat\mu,\mu)$ for distributions in unbounded Banach spaces, including higher dimensional Euclidean spaces ($\mathcal X=\mathbb R^d$ where $d$ may grow to $\infty$ as $n$ increases) and separable Hilbert spaces (square-integrable functions on $[0,1]$) as important examples.  Our argument combines the strengths from several recent works and provides more general and improved results in settings of practical interest. In addition to bounding $\mathbb E W_p(\hat\mu,\mu)$, we also establish a concentration inequality for $W_p(\hat\mu,\mu)$ around its expected value, which is dimension-free and offers improvements over existing results.


\subsection{Motivating examples.}\label{sec:motivating}
The convergence of empirical measure in Wasserstein distances in high-dimensional and functional spaces can be useful in different machine learning and statistical inference contexts.  Here we describe two examples in detail.  

\textit{Example 1. Unsupervised deep learning.} The first example is the Wasserstein GAN \citep{WGAN}, where GAN stands for ``generative adversarial network'' \citep{GAN}, a flexible framework for unsupervised learning.  The goal of GAN is to approximate an unknown distribution $\mu_r$ on a sample space $\mathcal X$, given an i.i.d sample from $\mu_r$. A typical example is the hypothetical distribution of all images.  To this end, GAN starts from a simple underlying distribution $\mu_z$ on a possibly different space $\mathcal Z$, such as the standard multivariate Gaussian in $\mathcal Z=\mathbb R^{100}$, and then finds a mapping $g:\mathcal Z\mapsto \mathcal X$, such that the distribution induced by $g(Z)$ on $\mathcal X$, denoted by $\mu_g$, is close to $\mu_r$.

In practice, neither $\mu_r$ nor $\mu_g$ is available. But the empirical version $\hat\mu_r$ is given as the data, and $\hat\mu_g$ can be generated by passing i.i.d samples from $\mu_z$ through the black-box function $g$. So one has to work with the empirical distributions instead of the original underlying distributions. In particular, Wasserstein GAN tries to find a mapping $g$ to minimize
$
W_1\left(\hat\mu_g,\hat\mu_r\right)
$
using deep neural networks.  We refer to the original papers \citep{GAN,WGAN} for the detailed realization and implementation of such a minimization problem.
Since its invention, Wasserstein GANs have found great success in image data due to the nice behavior of Wasserstein distances between mutually singular distributions. 

However, in machine learning, we are ultimately interested in the ability to approximate/predict future observations generated from the same underlying distribution. This leads to the problem of \emph{generalization}.  In the context of Wasserstein GAN, this problem further reduces to controlling
$W_1(\mu_g,\mu_r)$.   Given that the WGAN algorithm finds a $g$ such that $
W_1\left(\hat\mu_g,\hat\mu_r\right)
$ is small, one would know $W_1(\mu_g,\mu_r)$ is small if both $W_1(\mu_g,\hat\mu_g)$ and $W_1(\mu_r,\hat\mu_r)$ are small.  
Since WGAN is typically applied to images and videos, it is desirable to cover cases where $\mathcal X$ is high-dimensional or infinite-functional.

\textit{Example 2. Latent space representation of exchangeable random graphs.}
Exchangeable random graphs \citep{Aldous81,Hoover82,Kallenberg89} have been an important probabilistic tool in network research \citep{BickelC09,Lovasz12}. The existing parameterization of such exchangeable random graphs, called the \emph{graphon} \citep{Lovasz12}, has some undesirable identifiability issues, making the corresponding statistical inferences challenging \cite{Wolfe13,Airoldi13,GaoLZ15}.
Recently, \cite{Lei18network} established a new parameterization for exchangeable random graphs to facilitate the statistical inferences for network data. There are three findings in \cite{Lei18network}.
\begin{enumerate}
  \item [(i)] Under certain mild regularity conditions, an exchangeable random graph $\mathcal A$ can be parameterized by a distribution $\mu_{\mathcal A}$ on a separable Hilbert space $\mathcal H$.
  \item [(ii)] For two exchangeable random graphs $\mathcal A_1$ and $\mathcal A_2$ with corresponding Hilbert space distribution parameters $\mu_1$ and $\mu_2$, the difference in graph limits between $\mathcal A_1$ and $\mathcal A_2$ (also known as the cut-distance) can be upper bounded by $W_1(\mu_1,\mu_2)$.
  \item [(iii)] Suppose we observe a partial realization of $\mathcal A$ with $n$ nodes, then each node can be embedded in the Hilbert space $\mathcal H$ as an i.i.d sample from $\mu_{\mathcal A}$, and it is possible to find a good approximation of $\hat\mu_{\mathcal A}$ in Wasserstein-1 distance using the observed random graph with $n$ nodes. Here $\hat\mu_{\mathcal A}$ is the empirical version of $\mu_{\mathcal A}$ with $1/n$ probability mass at each of the embedded nodes.
\end{enumerate}
Combining the above three claims, one would be able to recover $\mu_{\mathcal A}$ in Wasserstein-1 distance from a partial realization of the exchangeable random graph $\mathcal A$ if one can show that $W_1(\hat\mu_{\mathcal A},\mu_{\mathcal A})$ is small.  Again, given the context we wish to consider this problem in an infinite-dimensional separable Hilbert space without boundedness assumptions.

\subsection{Overview of results.}
In \Cref{sec:generic}, we develop a general argument for bounding $\mathbb E W_p(\hat\mu,\mu)$. This approach combines the hierarchical transport program under moment conditions developed in \cite{DereichSS13,FournierG15}, with the abstract partition approach in \cite{WeedB17}. Comparing to \cite{DereichSS13,FournierG15}, both of which focus on fixed dimensionality, our argument uses covering numbers of a baseline set that is compatible with the corresponding moment condition, leading to improved rates of convergence in high dimensional settings, as well as applicability to general Banach spaces.  Comparing to \cite{WeedB17,SinghP18}, both of whom focus on bounded spaces, our result extends to the unbounded case using a moment condition.  The result of \cite{WeedB17} is for fixed dimensionality, while the infinite dimensional example in \cite{SinghP18} is a special case of the moment condition used in this paper.  An example that cannot be covered by these existing works is infinite dimensional Gaussian processes.

In \Cref{sec:euclidean} we demonstrate the advantage of abstract partition in Euclidean spaces with large dimensionality.  If $\mathcal X=\mathbb R^d$ and $\mathbb E_{X\sim \mu}\|X\|^q\le M_q^q$ for some $p<q\ll d\ll \log n$.  Then (see \Cref{thm:euclidean}) 
  $$
\mathbb E W_p(\hat\mu,\mu)\le c_{p,q} M_q n^{-1/d}\,,
  $$
where the constant $c_{p,q}$ depends only on $(p,q)$ but not $d$.  In contrast, the results in \cite{DereichSS13,FournierG15,WeedB17,SinghP18} have a constant as an unbounded function of $d$.

In \Cref{sec:function} we further apply the general recipe to distributions on separable Hilbert spaces under ellipsoid-type moment conditions. Let $\mu$ be a distribution on $\mathcal X=\{x\in \mathbb R^\infty: \sum_{m\ge1}x_m^2<\infty\}$. Let $\tau=(\tau_m:m\ge 1)$ be a sequence of nonnegative numbers and $\rho_\tau(x)=[\sum_{m\ge 1}(x_m/\tau_m)^2]^{1/2}$. The moment condition we consider here is $\mathbb E\rho^q_\tau(X)=M_q^q<\infty$. 
We consider two popular cases of $\tau$ in functional data analysis. In the first case $\tau_m=m^{-b}$ for some $b>1/2$, which corresponds to a polynomial decay of the coordinates.  In this case we have (see \Cref{thm:poly}), for some constant $c_{p,q,b}$ depending only on $(p,q,b)$,
$$
\mathbb E W_p(\hat\mu,\mu)\le c_{p,q,b} M_q (\log n)^{-b}\,.
$$
In the second case $\tau_m=\gamma^{-m}$ for some $\gamma>1$, which corresponds to an exponential decay of the coordinates. In this case we have (\Cref{thm:exp}), for some constant $c_{p,q,\gamma}$ depending only on $(p,q,\gamma)$,
$$\mathbb E W_p(\hat\mu,\mu)\le c_{p,q,\gamma} M_q e^{-\sqrt{2\log\gamma\log n}}\,.$$
These upper bounds are rate-optimal because in each case we can construct a corresponding $\mu$ such that the upper bound is matched by a lower bound with the same rate.  

In \Cref{sec:concentration}, we establish concentration inequalities for $W_p(\hat\mu,\mu)$. Assuming a sub-exponential tail behavior of $\|X\|$, the rate of concentration of $W_p(\hat\mu,\mu)$ around its expected value can be as small as $n^{-1/2}$ when $p=1$, which can vanish faster than the optimal upper bound of $\mathbb E W_p(\hat\mu,\mu)$ in high-dimension settings.  The new concentration inequality is dimension-free. In the case of $p=1$, it provides a better concentration rate compared to \cite{Boissard11} and \cite{FournierG15} under the sub-exponential tail condition.  We also extend an argument of \cite{BobkovL14,GozlanC10} from the case of $\mathcal X=\mathbb R^1$ to more general cases.  Our  concentration results are based on a variation of McDiarmid's inequality and the Lipschitz property of Wasserstein distances.

\paragraph{Notation.}
$\mathbb B_x^r=\{y\in\mathcal X:\|x-y\|\le r\}$ is the ball of radius $r$ centered at $x$.\\
$\mathbb N$ is the set of all positive integers.\\
For a random variable $Z$, $\|Z\|_p=(\mathbb E|Z|^p)^{1/p} $ denotes the $L_p$ norm of $Z$.\\
$\|x\|$ denotes the norm of $x$ in the Banach space $\mathcal X$.\\
$L_2$ denotes the separable Hilbert space $\{x\in\mathbb R^\infty:\sum_{m\ge 1}x_m^2<\infty\}$.\\
For $B\subseteq\mathcal X$ and $a\in\mathbb R$: $aB=\{ax:x\in B\}$.\\
For a set $B$ and a measure $\mu$, $\mu|_B$ is the confinement of $\mu$ on $B$: $(\mu|_B)(A)=\mu(B\cap A)$.

\section{A general upper bound}\label{sec:generic}
Given a distribution $\mu$ on an unbounded space $\mathcal X$ and its empirical version $\hat\mu$ of sample size $n$. The Wasserstein distance $W_p(\hat\mu,\mu)$ is determined by essentially two aspects of the problem: the complexity of the support of $\mu$ and the tail behavior of $\mu$. We describe these two aspects in detail and introduce the corresponding techniques.

\subsection{Covering and partition in $\mathcal X$.}
The first quantity that affects $W_p(\hat\mu,\mu)$ comes from the dimensionality (or complexity) of the space $\mathcal X$.  Suppose for now that $\mu$ is supported on a bounded subset $B_0\subset\mathcal X$.  The complexity of $B_0$ can be described by its $\epsilon$-covering number.

\begin{definition}[$\epsilon$-covering]
  For $B\subseteq\mathcal X$, and $\epsilon \in (0,1]$, the $\epsilon$-covering number, denoted as $N_\epsilon(B)$ is
  \begin{align*}
    N_\epsilon(B)=\inf\left\{N\in \mathbb N:~\exists~x_1,...,x_N\in B,\text{ s.t. }B\subseteq \cup_{i=1}^N \mathbb B_{x_i}^\epsilon\right\}\,.
  \end{align*}
\end{definition}

$\log N_\epsilon(B)$ is called the $\epsilon$-entropy number of $B$ and the speed at which it increases as $\epsilon$ decreases plays an important role in the empirical process theory.  

Given an $\epsilon\in(0,1]$, let $x_1,...,x_{N_\epsilon(B_0)}$ be the centers of an $\epsilon$-covering of $B_0$. Let $F_1=\mathbb B_{x_1}^\epsilon\cap B_0$, and $F_i=(\mathbb B_{x_i}^\epsilon \cap B_0)\backslash (\cup_{1\le j<i} F_j)$ for $2\le i\le N_\epsilon(B_0)$, then $\mathcal A=\{F_i:1\le i\le N_\epsilon(B_0)\}$ is a partition of $B_0$ satisfying $\max_{F\in\mathcal A} {\rm diam}(F)\le 2\epsilon$ and ${\rm card}(A)\le N_\epsilon(B_0)$.

A basic tool for controlling $W_p(\hat\mu,\mu)$ is to construct a sequence of nested $\epsilon_\ell$-partitions for a sequence of geometrically decaying $(\epsilon_\ell:\ell \in 0\cup \mathbb N)$.   \cite{DereichSS13,FournierG15} use $B_0=[-1,1]^d$ and recursive dyadic partitions of $B_0$. The use of general $\epsilon$-partitions first appeared in \cite{WeedB17}.

In our proof of upper bound, we consider $\epsilon_\ell=3^{-(\ell+1)}$ as in \cite{WeedB17} and define
\begin{equation}\label{eq:bar_N}
\bar N_\ell(B_0)=N_{3^{-(\ell+1)}}(B_0)\,,~~\ell\in 0\cup \mathbb N\,.
\end{equation}
The complexity of $B_0$ will be reflected in the upper bound of $W_p(\hat\mu,\mu)$ through $\bar N_\ell$. 

The following lemma is a basic result in upper bounding $W_p(\mu,\nu)$ in the bounded support case using a hierarchical partition.
\begin{lemma}\label{lem:bounded}
  Let $\mu$ and $\nu$ be two probability measures supported on $B_0$.
  Let $\ell^*\in \{0\}\cup\mathbb N$ and $\{B_0\}=\mathcal A_0,\mathcal A_1,...,\mathcal A_{\ell^*}$ be a sequence of nested partitions of $B_0$ such that $\sup_{F\in\mathcal A_\ell}{\rm diam}(F)\le c3^{-\ell}$ for some constant $c$, then there exists a constant $c_p$ depending only on $p$ such that
  $$
  W_p^p(\mu,\nu)\le c_{p}\left(3^{-p\ell^*}+\sum_{\ell=0}^{\ell^*}
  3^{-p\ell}\sum_{F\in\mathcal A_\ell}|\nu(F)-\mu(F)|\right)\,.
  $$
\end{lemma}
\Cref{lem:bounded} first appears in its general form as Proposition 1 of \cite{WeedB17}, which extends the Euclidean space version with dyadic hypercube partition in \cite{DereichSS13,FournierG15}.  In \Cref{app:proof_gen}, we give an alternative proof of it, adapting the elegant Markov chain argument in \cite{DereichSS13}.

In \Cref{lem:bounded}, the sum over $\ell$ comes from the fact that the best way to transport the mass in $\mu$ to $\hat\mu$ is through its blurred versions on the $\epsilon_\ell$-partition.  Let $\mathcal A_\ell$ be the partition of $B_0$ at level $\ell$, the blurred version, $\mu_\ell$, of $\mu$ is (using the convention $0/0=0$)
\begin{align*}
\mu_\ell=\sum_{F\in\mathcal A_\ell} \mu_\ell|_F\,,~~\text{with}~~
\mu_\ell|_F = \frac{\hat\mu(F)}{\mu(F)}\mu|_F.
\end{align*}
These blurred versions can be used to construct a chain of transports that link between $\mu$ and $\hat\mu$, while each having small transport cost.

\subsection{Telescope decomposition and moment condition.}\label{sec:telescope}
Now we drop the assumption that $\mu$ is supported on a bounded set $B_0$.  The second quantity that affects $W_p(\hat\mu,\mu)$ comes from the tail behavior of $\mu$, which we control using a ``telescope decomposition'' of $\mathcal X$, combined with a moment condition.

Let $\rho:\mathcal X\mapsto \{0\}\cup\mathbb R^+$ be a homogenous functional satisfying
\begin{enumerate}
  \item [(\textbf{R1})] $\rho(a x)=|a| \rho(x)$ for all $a\in \mathbb R$;
  \item [(\textbf{R2})] $B_{0,\rho}\stackrel{\rm def}{=}\{x: \rho(x)\le 1\}\subseteq \mathbb B_0^1$.
\end{enumerate}
Condition (R1) simply requires homogeneity of $\rho$. Condition (R2) ensures that $\rho$ dominates the Banach space norm.

When $\mathcal X=\mathbb R^d$, an example of $\rho$ is $\rho(x)=\|x\|$ and the corresponding $B_{0,\rho}$ is the  unit Euclidean ball. When $\mathcal X=L_2$, an example of $\rho$ is $\rho_\tau(x)=[\sum_{m\ge 1}(x_m/\tau_m)^2]^{1/2}$ for some $\tau\in (\mathbb R^+)^\infty$ satisfying $\tau_m\le 1$ for all $m$.

Sometimes the dimensionality of $\mathbb B_0^1$ is too high to yield any meaningful upper bound. For example, if $\mathcal X=L_2$ then $N_\epsilon(\mathbb B_0^1)=\infty$ for all values of $\epsilon$ smaller than a constant.  However,  $B_{0,\rho}$ can be substantially smaller than $\mathbb B_0^1$ for appropriate choices of $\rho$, so that we can still obtain useful upper bounds.

To deal with unbounded support of $\mu$, we decompose $\mu$ to its confinements on a sequence of ``telescoping'' subsets $(B_{j}:j\in 0\cup\mathbb N)$, where $B_0=B_{0,\rho}$ and
\begin{equation*}
  B_{j}= (2^jB_{0})\backslash B_{j-1}, ~~\forall~j\in\mathbb N\,.
\end{equation*}
The tail behavior of $\mu$ is characterized by $\mu|_{B_j}$ for large values of $j$. Our analysis will consider $\tilde\mu_j$, the induced measure of $\frac{\mu|_{B_j}}{\mu(B_j)}$ after the mapping $x\mapsto 2^{-j}x$, and apply the bounded support results to $\tilde \mu_j$.

The following ``telescoping lemma'' generalizes the corresponding argument in \cite{FournierG15} to the case of general Banach spaces and general moment conditions.
\begin{lemma}\label{lem:telescope} Let $\rho$ be a function that satisfies (R1) and (R2).
Let $\mu$ and $\nu$ be two probability measures on $\mathcal X$ such that $\rho(x)<\infty$ almost surely under $\mu$ and $\nu$. For $j\ge 0$, define $\tilde\mu_j$ ($\tilde\nu_j$) to be the normalized probability measure induced by $\mu|_{B_j}$ ($\nu|_{B_j}$) after the mapping $x\mapsto 2^{-j}x$. Then there exists a constant $c_p$ depending only on $p$ such that \begin{equation}\label{eq:telescope}
  W_p^p(\mu,\nu)\le c_p
  \sum_{j\ge 0}2^{pj}\left[(\mu(B_j)\wedge\nu(B_j)) W_p^p(\tilde\mu_j,\tilde\nu_j)+|\mu(B_j)-\nu(B_j)|\right]\,.
\end{equation}  
\end{lemma}
The exponentially increasing factor $2^{pj}$ in the above sum needs to be offset by the small terms $\mu(B_j)\wedge\nu(B_j)$. This motivates the following moment condition:
\begin{equation}\label{eq:q_moment}
  \|\rho(X)\|_q = M_q <\infty
\end{equation}
for some constant $q>p$.

If the moment condition \eqref{eq:q_moment} holds with $M_q=1$ for $X\sim\mu$, then the contribution of $W_p^p(\tilde\mu_j,\tilde\nu_j)$ in \eqref{eq:telescope} becomes less important as $j$ gets larger, because
$\mu(B_j)\le 2^{-q(j-1)}$.  

The moment condition $\|\rho(X)\|_q= M_q<\infty$ for $X\sim\mu$ serves another purpose: It implies that $\rho(x)<\infty$ $\mu$-almost surely.  In this paper we are mainly interested in the case $\nu=\hat\mu$, whose support is no larger than that of $\mu$.  As a result, we have $\rho(x)<\infty$ almost surely under $\mu$ and $\nu$, as required by \Cref{lem:telescope}. An important consequence is that $\{B_j:j\ge 0\}$ forms a partition of the supports of $\mu$ and $\nu$.

\subsection{A general upper bound.}
\begin{theorem}\label{thm:general}
  Let $\bar N_\ell$ be defined as in \eqref{eq:bar_N} with $B_0=B_{0,\rho}$ for a function $\rho$ satisfying conditions R1 and R2.
  Assume the moment condition \eqref{eq:q_moment} holds with $q>p\ge 1$ for $X\sim \mu$, then for any $\ell^*\in\mathbb N$, we have
  \begin{equation}\label{eq:thm_gen}
    \mathbb E W_p^p(\hat\mu,\mu)\le c_{p,q}M_q^p\sum_{j\ge 0}2^{pj}\left\{2^{-qj}3^{-p\ell^*}+\sum_{\ell=0}^{\ell^*}3^{-p\ell}\left[2^{-qj}\wedge\left(\bar N_\ell 2^{-qj}n^{-1}\right)^{1/2}\right]\right\}\,.
  \end{equation}
\end{theorem}
\Cref{thm:general} looks similar to its counterparts in \cite{DereichSS13,FournierG15,WeedB17,SinghP18}, and has combined advantages from them.  Comparing to \cite{DereichSS13,FournierG15}, \eqref{eq:thm_gen} uses a finite sum over $\ell$ which allows for the use of customized $\epsilon$-partitions. We shall see in \Cref{sec:euclidean} that this can improve the dependence on the Euclidean dimensionality $d$ when $\mathcal X=\mathbb R^{d}$.  This also allows us to cover the case of $\mathcal X=L_2$ with appropriately chosen $\rho$ functions, as detailed in \Cref{sec:function}.  On the other hand, \cite{WeedB17,SinghP18} do not have the outer sum over $j$ and only have the term for $j=0$, which only works for bounded support. We prove \Cref{thm:general} in \Cref{app:proof_gen}. 

The details in \eqref{eq:thm_gen} reflect the trade-off between the depth of the partition (indexed by $\ell$), the sample size $n$, and the importance of each telescoping layer $B_j$.
The most intriguing part of \eqref{eq:thm_gen} is the term $2^{-qj}\wedge\left(\bar N_\ell 2^{-qj}n^{-1}\right)^{1/2}$, where the ``$\wedge$'' comes from the fact $\mathbb E|\mu(F)-\hat\mu(F)|\le (2\mu(F))\wedge (\mu(F)/n)^{1/2}$ for any $F$.  It suggests that for different layers $B_j$, the critical partition levels $\ell_j^*$ (the value of $\ell$ such that $2^{-qj}\asymp \left(\bar N_\ell 2^{-qj}n^{-1}\right)^{1/2}$) will be different.

\section{Euclidean spaces: optimal dependence on $d$}\label{sec:euclidean}
In this section we assume $\mathcal X=\mathbb R^d$, together with $\rho(x)=\|x\|$, the Euclidean norm.

\begin{theorem}\label{thm:euclidean}
  If $\mathcal X=\mathbb R^d$, $\rho(x)=\|x\|$ and $\|X\|_q= M_q<\infty$ for $X\sim\mu$ and some $q>p\ge 1$, then
  \begin{align}\label{eq:thm_euclidean}
    \mathbb E W_p(\hat\mu,\mu)\le c_{p,q} M_q n^{-\left[\frac{1}{(2p)\vee d}\wedge (\frac{1}{p}-\frac{1}{q})\right]} (\log n)^{\zeta_{p,q,d}/p}\,,
  \end{align}
  where $c_{p,q}$ is a constant depending only on $(p,q)$ (not $d$), and
  $$
  \zeta_{p,q,d}=\left\{\begin{array}{ll}
    2 & \text{ if } d=q=2p\,,\\
    1 & \text{ if ``} d\neq 2p \text{ and }q=\frac{dp}{d-p} \wedge 2p \text{'' or ``} q>d=2p\text{''}\,,\\
    0 & \text{ otherwise.}
  \end{array}\right.
  $$
\end{theorem}

\begin{remark}
\Cref{thm:euclidean} can be compared with Theorem 1 of \cite{FournierG15} and Example 2 of \cite{SinghP18} (the latter assumes bounded support, which corresponds to the case of $q=\infty$).  The key difference is that our constant $c_{p,q}M_q$ does not depend on $d$, while in all existing similar results the constants depend on $d$ in an unbounded manner. Such an improved dependence on $d$ makes a qualitative difference when we want to project the infinite dimensional sample points onto a subspace whose dimension increases as the sample size $n$.  This is particularly the case in the latent space network representation example described in \Cref{sec:motivating}.
\end{remark}

\begin{remark}
  The logarithm factor in \eqref{eq:thm_euclidean} is active only in the boundary cases such as $d=2p$ or $q=2p$, and is upper bounded by $(\log n)^{2/p}$ in the worst case.  For the cases without such a logarithm factor, the optimality of the main factor $n^{-\left[\frac{1}{(2p)\vee d}\wedge (\frac{1}{p}-\frac{1}{q})\right]}$ has been discussed in detail in \cite{FournierG15}.
  \end{remark}

We give the proof of \Cref{thm:euclidean} in \Cref{app:special_cases}. To highlight the novel ingredients in our proof, here we consider the case $d>2p$. 

\textit{Suboptimality of hypercube partition.}
If we follow the proof of \cite{FournierG15} that uses $B_0=[-1,1]^{d}$ and $\mathcal A_\ell$ as dyadic partition of $B_0$ with hypercubes of side length $2^{1-\ell}$, then the constant factor in \eqref{eq:thm_euclidean} will have a factor of $\sqrt{d}$ that comes from the diameter of $B_0$.  

\textit{Covering numbers of the unit Euclidean ball.}
To avoid such a $\sqrt{d}$ factor we can choose $B_0=\mathbb B_0^1$, the $d$-dimensional unit ball centered at $0$.  However, the best upper bound for the covering number of $\mathbb B_0^1$ is $N_\epsilon(\mathbb B_0^1)\le \epsilon^{-d}(d+1)^c$ for some constant $c\ge 3/2$ \citep[][Theorem 3.1]{Verger-Gaugry05}.  Thus we need to work with the upper bound $\bar N_\ell\le 3^{(\ell+c_0)d}$ for some different positive constant $c_0>1$. Directly factoring out the $3^{c_0d/2}$ factor from the $\bar N_\ell^{1/2}$ term in \eqref{eq:thm_gen} will bring the undesirable (and indeed unnecessary) $3^{c_0d/2}$ factor to the final upper bound.

Our proof avoids this by directly keeping track of each $\bar N_\ell$. The effect of $d$ shows up differently in \eqref{eq:thm_gen} for small and large value of $j$'s, which are controlled by different techniques.
Here we briefly describe the case when $j$ is small. When $d>2p$, for small values of $j$ the main contributing term in the inner sum of \eqref{eq:thm_gen} is $3^{-p\ell} \bar N_\ell^{1/2}2^{-qj/2}n^{-1/2}$, which is an increasing geometric sequence and hence bounded by the last term. The key observation here is that the last term in this sequence is bounded by $3^{-p\ell}2^{-qj}$ by construction, and hence there is no need to use the detailed formula for $\bar N_\ell$ as a function of $(\ell,d)$.  This is indeed a benefit of the abstract partition scheme where the diameter of covering balls is dimension-free.
%

\section{Functional spaces}\label{sec:function}
In this section we consider the separable Hilbert space $\mathcal X=L_2=\{x\in \mathbb R^\infty:\sum_{m=1}^\infty x_m^2<\infty\}$.  This functional space is isomorphic to $L_2([0,1])$, the space of square integrable functions on $[0,1]$, which is a commonly assumed data space in functional data analysis.

Now let $\mu$ be a probability measure on $L_2$. As discussed in \Cref{sec:telescope}, we can use an ellipsoidal moment condition on $\mu$ to deal with the infinite dimensionality.  
The function $\rho$ used for our moment conditions has the following form
\begin{align*}
  \rho_\tau(x)=\left[\sum_{m=1}^\infty \left(\frac{x_m}{\tau_m}\right)^2\right]^{1/2}\,,
\end{align*}
where $\tau=(\tau_m:m\ge 1)$ is a sequence of positive numbers with $\sup_{m}\tau_m\le 1$.

The moment condition \eqref{eq:q_moment} becomes stronger if $\tau_m$'s are small. We consider two types of such $\tau$ sequences,
 which correspond to different speeds at which $\tau_m$ vanishes.  Relevance of these scenarios in functional data analysis is discussed in \Cref{sec:fpc}.
 
We adopt a minimax perspective to illustrate the optimality of the upper bounds. Consider a class of distributions with a specified tail behavior, if we can show that a lower bound of the same rate is achieved by at least one member in this class, then the upper bound rate cannot be improved uniformly within such a class of distributions.

\subsection{Polynomial decay.}\label{sec:poly}
In this case we consider $\tau$ sequences with 
\begin{equation}\label{eq:poly_tau}
  \tau_m=m^{-b}\,,~~\forall~m\ge 1
\end{equation} for some $b> 1/2$.

The class of distributions of interest are those satisfy the moment condition \eqref{eq:q_moment} with $\rho=\rho_\tau$ and $\tau$ given in \eqref{eq:poly_tau}.
In this case $\rho_\tau(x)=[\sum_{m}(m^b x_m)^2]^{1/2}$. Formally,
define the distribution class
\begin{equation}\label{eq:class_poly}
  \mathcal P_{\rm poly}(q,b,M_q):=\left\{\mu:~\mathbb E_{X\sim \mu}\left[\sum_{m=1}^\infty \left(m^b X_m\right)^2\right]^{\frac{q}{2}}\le M_q^q\right\}\,.
\end{equation}

\begin{theorem}\label{thm:poly}
  If $p,q,b$ are constants such that $1\le p <q$ and $b>1/2$, then there exist positive constants $\underline{c}_{p,q,b}$, $\bar{c}_{p,q,b}$ depending on $(p,q,b)$ such that
  \begin{align*}
  \underline{c}_{p,q,b} M_q (\log n)^{-b} \le \sup_{\mu\in\mathcal P_{\rm poly}(q,b,M_q)}\mathbb E W_p(\hat\mu,\mu)\le \bar{c}_{p,q,b}M_q(\log n)^{-b}\,.
  \end{align*}
\end{theorem}

The poly-log rate in \Cref{thm:poly} comes from the fact that the metric entropy (logarithm of the covering number) of $B_{0,\rho_\tau}$ has a polynomial dependence on $\epsilon^{-1}$ and hence an exponential dependence on $\ell$ if $\epsilon=3^{-(\ell+1)}$. This is in sharp contrast with the Euclidean case, where the metric entropy is a linear function of $\log(\epsilon^{-1})$ and hence a linear function of $\ell$ if $\epsilon=3^{-(\ell+1)}$.

One can compare \Cref{thm:poly} with Example 4 of \cite{SinghP18}, which independently provides similar matching upper and lower bounds for a different function class.  To make them directly comparable, we consider a typical special case where $\mathcal X$ consists of Lipschitz functions on $[0,1]$ satisfying $f(0)=f(1)$. Such an $\mathcal X$ is a H\"{o}lder class, so example 4 of \cite{SinghP18} suggests an upper bound of order $(\log n)^{-1}$. On the other hand, the H\"{o}lder class is a subset of $\mathcal P_{\rm poly}(q,b,M_q)$ with $b=1$ under the trigonometric basis, which is a Sobolev ellipsoid (see Chapter 1.7.1 of \cite{Tsybakov09}).  Therefore, \Cref{thm:poly} also implies an upper bound of order $(\log n)^{-1}$.

\subsection{Exponential decay.}\label{sec:exp}
In this case we consider $\tau$ sequences with exponential decay:
\begin{equation}\label{eq:exp_tau}
  \tau_m=\gamma^{-(m-1)}\,,~~\forall~m\ge 1
\end{equation}
for some $\gamma>1$.

The corresponding class of distributions satisfy the moment condition \eqref{eq:q_moment} with $\rho=\rho_\tau$ and $\tau$ given in \eqref{eq:exp_tau}. In this case $\rho_\tau(x)=[\sum_{m}(\gamma^{m-1}x_m)^2]^{1/2}$.
Define the distribution class
\begin{equation}\label{eq:class_exp}
  \mathcal P_{\rm exp}(q,\gamma,M_q):=\left\{\mu:~\mathbb E_{X\sim \mu}\left[\sum_{m=1}^\infty \left(\gamma^{m-1} X_m\right)^2\right]^{\frac{q}{2}}\le M_q^q\right\}\,.
\end{equation}

\begin{theorem}\label{thm:exp}
  If $p,q,\gamma$ are constants such that $1\le p<q$ and $\gamma>1$, then there exist positive constants $\underline{c}_{p,q,\gamma}$, $\bar{c}_{p,q,\gamma}$ depending only on $(p,q,\gamma)$ such that
  \begin{align*}
    \underline{c}_{p,q,\gamma}M_qe^{-\sqrt{2\log\gamma\log n}}\le
     \sup_{\mu\in\mathcal P_{\rm exp}(q,\gamma,M_q)}\mathbb E W_p(\hat\mu,\mu)
    \le \bar{c}_{p,q,\gamma}M_qe^{-\sqrt{2\log\gamma\log n}}\,.
  \end{align*}
\end{theorem}
The lower and upper bounds in \Cref{thm:exp} seem to be new in the literature. \Cref{thm:exp} shows that, even when the coordinates decay exponentially, the functional space is still fundamentally different from Euclidean spaces. As we will see in \Cref{lem:metric_entropy_exp} below, the metric entropy is a quadratic, instead of linear, function of $\log(\epsilon^{-1})$, so we can only have a sub-polynomial rate of convergence in \Cref{thm:exp}.  Moreover, the exponential decay makes the problem qualitatively different from the polynomial decay, as this sub-polynomial rate is faster than the poly-log rate in \Cref{thm:poly}.

In both \Cref{thm:poly,thm:exp}, our proofs reflect that when $\mathcal X$ has higher dimensionality than Euclidean space, $W_p(\hat\mu,\mu)$ is determined by the value $\epsilon$ such that $N_\epsilon(B_0)\asymp n$.

The lower and upper bounds of the covering numbers $N_\epsilon(B_0)$ and $\bar N_\ell$ required in the proof of \Cref{thm:exp} are provided in the following lemma, which adapts the covering number bounds for finite dimensional ellipsoids in \cite{DumerPP04,Dumer06} to infinite dimensional ellipsoids using a truncation argument. Equivalent forms of this result have appeared in the literature \citep[see, for example,][Theorem 7]{Williamson01}.  The version presented here is more direct applicable for our purpose.
\begin{lemma}\label{lem:metric_entropy_exp}
  Let $B_0=\{x\in L_2:\rho_\tau(x)\le 1\}$ with $\tau$ given in \eqref{eq:exp_tau}.  Then for $\epsilon\in(0,1]$,
  \begin{align*}
    \log N_\epsilon(B_0)\ge \frac{\left[\log(\epsilon^{-1})\right]^{2}}{2\log\gamma}
  \end{align*}
  and for $\ell\in \mathbb N$,
  \begin{align*}
    \log \bar N_\ell = \log N_{3^{-(\ell+1)}}(B_0) \le c_\gamma(\ell+c_1)^2+\log c_0
  \end{align*}
  with $c_\gamma=\frac{(\log 3)^2}{2\log\gamma}$ and $c_0,c_1$ positive constants depending on $\gamma$.
\end{lemma}

\subsection{Functional principal components.}\label{sec:fpc}
In the previous two subsections we see that $W_p(\hat\mu,\mu)$ tends to be smaller if the majority probability mass of $\mu$ can be covered by an ellipsoid $B_0=B_{0,\rho_\tau}$, with fast decaying axes $\tau_m$.  Intuitively speaking, this corresponds to an orthogonal basis of $L_2$, denoted as $\{\phi_m:m\ge 1\}$ such that the variability of the coordinate projections $\langle X, \phi_m \rangle$ decays fast when $X\sim \mu$.  Such an explanation suggests an immediate connection to functional principal components analysis \citep{HallMW06,Ramsay06}, where one can choose $\{\phi_m:m\ge 1\}$ to be the eigenvectors of the covariance operator of $\mu$.

To facilitate discussion, we assume that $\mathbb E_{X\sim \mu}X=0$ and $\mathbb E_{X\sim \mu}\|X\|^2<\infty$. Then we can write the random function (in a separable Hilbert space, such as $L_2$) $X\sim \mu$ in its Karhunen-Lo\`{e}ve decomposition
\begin{equation}\label{eq:KL-decomp}
  X = \sum_{m=1}^\infty \sigma_m Z_m \phi_m\,,
\end{equation}
where $\sigma_1\ge \sigma_2\ge ... \ge 0$ are the standard deviations of the coordinate projections of $X$ ranked in decreasing order, and $(Z_m:m\ge 1)$ are uncorrelated random variables with mean $0$ and unit variance.  Standard theory shows that such a basis $(\psi_m:m\ge 1)$ leads to the fastest decay of the tail squared sum $\sum_{j\ge m}\sigma_j^2$ for all $m$.

The speed at which $\sigma_m$ decays as $m$ increases measures the ``regularity'' of the functional data distribution $\mu$.  Here we can make explicit connections between the eigen-decay of $\mu$ and the ellipsoidal moment conditions considered in the previous two subsections.

\begin{proposition}\label{pro:fpc} If the standardized principal component scores $(Z_m:m\ge 1)$ of $\mu$ satisfy $\sup_m \|Z_m\|_q=M<\infty$ with $q\ge 2$, then the following hold.
  
1. (Polynomial decay) If $$\sigma_m\le c_0 m^{-(b_0+1/2)}$$
for some constants $b_0>1/2$, $c_0>0$, then the moment condition \eqref{eq:q_moment} holds with $\tau$ sequence given in \eqref{eq:poly_tau} for any constant $b\in(1/2,~b_0)$ and $M_q=c_{b,b_0,q}c_0 M$, where $c_{b,b_0,q}$ is a constant depending only on $(b,b_0,q)$.

2. (Exponential decay) If $$\sigma_m\le c_0 \gamma_0^{-(m-1)}$$ for some constants $\gamma_0>1$, $c_0>0$, then the moment condition \eqref{eq:q_moment} holds with $\tau$ sequence given in \eqref{eq:exp_tau} for any $\gamma\in (1,\gamma_0)$ and $M_q=c_{\gamma,\gamma_0,q}c_0 M$ for some constant $c_{\gamma,\gamma_0,q}$ depending only on $(\gamma,\gamma_0,q)$.
\end{proposition}

\section{Concentration inequalities}\label{sec:concentration}
In this section we consider the concentration of measure for $W_p(\hat\mu,\mu)$.  In particular, we are interested in the following two types of concentration inequalities.
\begin{enumerate}
  \item \emph{Mean-concentration}. The goal is to find good upper bounds of the probability $$
\mathbb P\left[|W_p(\hat\mu,\mu)-\mathbb E W_p(\hat\mu,\mu)|>t\right]
$$
for $t\in\mathbb R^+$.
\item \emph{$0$-concentration}. The goal is to upper bound
$$
\mathbb P\left[W_p(\hat\mu,\mu)>t\right]\,.
$$
\end{enumerate}

\cite{BolleyGC07,FournierG15} have considered $0$-concentration of $W_p^p(\hat\mu,\mu)$.  We argue that concentration around the mean can be more informative, since in some important cases (for example, $p=1$ and $d$ large or infinite) the rate of concentration around the mean can be much smaller than the mean value itself.  Nevertheless, a good mean-concentration can lead to a good $0$-concentration, because if $\mathbb E W_p(\hat\mu,\mu)\le R_n$ then
$$
\mathbb P\left[W_p(\hat\mu,\mu)>t\right]\le \mathbb P\left[W_p(\hat\mu,\mu)-\mathbb E W_p(\hat\mu,\mu)>(t-R_n)_+\right]\,.
$$

\subsection{A Bernstein-type McDiarmid's inequality.}
Let $X_i$ ($1\le i\le n$) be independent (not necessarily identically distributed) samples from probability distributions $\mu_i$ on spaces $\mathcal X_i$, and $X_1',...,X_n'$ be independent copies of each $X_i$.  Denote $X=(X_1,...,X_n)$ and $X_{(i)}'=(X_1,...,X_{i-1},X_i',X_{i+1},...,X_n)$, which is identical to $X$ except that the $i$th entry is replaced by $X_i'$.
Let $f:\prod_{i=1}^n\mathcal X_i \mapsto \mathbb R$ be a function such that $\mathbb E |f(X)|<\infty$, and define
$$
D_i = f(X)-f(X_{(i)}')\,.
$$

To establish concentration around the mean for a function of independent random inputs, a classical tool is McDiarmid's inequality, which requires the difference $D_i$ to be bounded with probability one.  This is not enough for our purpose if we take $f(X)=W_p(\hat\mu,\mu)$ and the support of $\mu$ is unbounded.  To fix this problem, we need a stronger version of McDiarmid's inequality that holds under weaker conditions.  In particular, we consider the following Bernstein-type moment condition.
\begin{equation}\label{eq:bernstein}
\exists~ \sigma_i,M>0 ~\text{ s.t. }~  \mathbb E(|D_i|^k\mid X_{-i}) \le \frac{1}{2}\sigma_i^2 k! M^{k-2}\,,~~\forall~~\text{integer }k\ge 2\,,~~\text{a.s.}\,,
\end{equation}
where $X_{-i}=(X_1,...,X_{i-1},X_{i+1},...,X_n)$ is the random vector obtained by removing $X_i$ from $X$.

When $f(X)=\sum_{i=1}^n X_i$, then equation \eqref{eq:bernstein} reduces to the classical Bernstein's tail condition, which is used to prove the Bernstein's inequality for sums of independent random variables.  The following theorem extends the Bernstein's inequality to general functions of independent inputs.

\begin{theorem}[Bernstein-type McDiarmid's Inequality]\label{thm:Bernstein-McDiarmid}
  Assume \eqref{eq:bernstein} holds, then
  \begin{equation}
    \mathbb P\left( f- \mathbb E f\ge t\right)\le \exp\left(-\frac{t^2}{2\sigma^2+2tM}\right)
  \end{equation}
  where $\sigma^2=\sum_{i=1}^n\sigma_i^2$\,.
\end{theorem}
\Cref{thm:Bernstein-McDiarmid} can be derived from the more general martingale concentration inequalities \citep[][Theorem 1.2A]{Pena99} (see also \cite{Geer95}). In \Cref{app:concentration}, we give a simple and short proof that is similar to the standard proofs of Bernstein's inequality and McDiarmid's inequality.

\begin{remark}
  Applying \Cref{thm:Bernstein-McDiarmid} to $-f$ leads to the same concentration bound for the event $f-\mathbb E f\le -t$.  Thus \Cref{thm:Bernstein-McDiarmid} implies the absolute deviation bound
  $$    \mathbb P\left( |f- \mathbb E f|\ge t\right)\le 2\exp\left(-\frac{t^2}{2\sigma^2+2tM}\right)\,.
$$
\end{remark}

\begin{remark}
We discuss a few related works.  \cite{Kontorovich14} extends McDiarmid's inequality to unbounded differences by assuming sub-Gaussianity of $D_i$.  It is well-known that sub-Gaussianity is strictly stronger than Bernstein's tail condition \citep{vdvWellner}. In fact, the Bernstein's tail condition is nearly equivalent to being sub-exponential. An unpublished manuscript \cite{Ying_Bern} extends McDiarmid's inequality to a Bernstein-type condition.  However, this result still assumes that the differences $D_i$ are uniformly bounded.
When $\mathcal X_1=...=\mathcal X_n=\mathcal X$ and $X_i$'s are iid, \Cref{thm:Bernstein-McDiarmid} re-discovers a result of \cite[][Proposition 3.1]{DedeckerF15}.
\end{remark}

\subsection{Concentration of $W_p(\hat\mu,\mu)$.}
  Now let $f(X)=W_p(\hat\mu,\mu)$, and $\hat\mu'$ the empirical measure given by $X_{(i)}'$.  Since the ordering of sample points does not matter in $W_p(\hat\mu,\mu)$, the particular choice of $i$ does not matter. By triangle inequality $|f(X)-f(X_{(i)}')|\le W_p(\hat\mu,\hat\mu')\le n^{-1/p}\|X_i-X_i'\|$.

  If we assume that
  \begin{equation}\label{eq:bernstein_norm}
  \mathbb E(\|X_i\|^k) \le \frac{1}{2}s^2 k!V^{k-2}\,,~~\forall~\text{integer }k\ge 2\,,
  \end{equation}
  for some constants $s$, $V$, then straight calculation shows that such $f(X)=W_p(\hat\mu,\mu)$ satisfies the Bernstein-type tail condition \eqref{eq:bernstein} with $\sigma_i = 2 s n^{-1/p}$ and $M=2V n^{-1/p}$.  Thus we have the following corollary.
  \begin{corollary}\label{cor:concentration}
    If $X_i\sim\mu$ satisfies \eqref{eq:bernstein_norm} with some constants $s$, $V$, then for all $t>0$
  \begin{align*}
    \mathbb P\left[\left|W_p(\hat\mu,\mu)-\mathbb E W_p(\hat\mu,\mu)\right|\ge t\right]\le 2\exp\left(-\frac{t^2}{8s^2n^{1-2/p}+4Vtn^{-1/p}}\right)\,.
  \end{align*}
  \end{corollary}

  \Cref{cor:concentration} is useful when $p\in[1,2)$, implying that $W_p(\hat\mu,\mu)-\mathbb EW_p(\hat\mu,\mu)=O_P(n^{-(1/p-1/2)})$.  In the most interesting case $p=1$, the rate becomes the desired $n^{-1/2}$.  This result is dimension-free, assuming only the tail condition \eqref{eq:bernstein_norm}.
Condition \eqref{eq:bernstein_norm} is implied by the sub-exponential tail condition. Following \cite{vdvWellner}, for a random variable $Z$ and real number $\alpha\ge 1$, define the Orlicz $\psi_\alpha$ norm of $Z$ as
  \begin{equation}\label{eq:subexp}
  \|Z\|_{\psi_\alpha}=\inf \left\{c>0:\mathbb E  e^{|Z/c|^\alpha}\le 2\right\}\,.
  \end{equation}
  If $\left\| \|X_i\| \right\|_{\psi_1}=C<\infty$,
    then \eqref{eq:bernstein_norm} holds with $V=C$ and $s=\sqrt{2}C$.
This condition is called ``sub-exponential'' because
$\left\| \|X_i\| \right\|_{\psi_1}<\infty$ if and only if $\mathbb P(\|X_i\|\ge t)\le c e^{-c' t}$ for some positive  constants $c,c'$.

\begin{remark}
In the special case of $p=1$, there have been related mean-concentration results using variants of McDiarmid's inequality.  For example, \cite{Boissard11} also provides $n^{-1/2}$ mean-concentration for $W_1(\hat\mu,\mu)$, but requires a stronger sub-Gaussian tail of $\|X_i\|$. In a more recent work, \cite{DedeckerF15} establishes a similar mean-concentration result applicable to the case $p=1$, under a condition that can be translated to a sub-exponential tail of $\|X_i\|$ in our context.
\end{remark}

\paragraph{Euclidean spaces.}
If $\mathcal X=\mathbb R^d$ with $d>2$, and  $\left\|\|X_i\|\right\|_{\psi_1}<\infty$, the sub-exponential condition implies finiteness of all moments, and \Cref{thm:euclidean} implies that $\mathbb EW_1(\hat\mu,\mu)\le c_{1}n^{-1/d}$.
Applying the one-sided version of \Cref{cor:concentration} we get a $0$-concentration bound,
\begin{align}
  \mathbb P\left[W_1(\hat\mu,\mu)\ge t\right]\le &\mathbb P\left[W_1(\hat\mu,\mu)-\mathbb E W_1(\hat\mu,\mu)\ge t-c_{1}n^{-1/d}\right]\nonumber\\
  \le & \exp\left[-c n(t-c_1 n^{-1/d})_+^2\right]\,,\label{eq:concentration_example1}
\end{align}
where $c$ is a constant depending only on $\left\|\|X_i\|\right\|_{\psi_1}$.
The bound given in \eqref{eq:concentration_example1} can be compared with \cite[][Theorem 2, equation (1)]{FournierG15}, which assumes a stronger condition $\left\|\|X_i\|\right\|_{\psi_\alpha}<\infty$ for some $\alpha>1$.  In the interesting regime of $t\asymp n^{-1/d}$ (on the same scale as the optimal upper bound of $\mathbb E W_1(\hat\mu,\mu)$), the result of \cite{FournierG15} only provides a probability bound of order $O(1)$, while \eqref{eq:concentration_example1} provides a bound of $\exp(-c(c_2-c_1)^2n^{(1-2/d)})$ for $t=c_2 n^{-1/d}$ with $c_2>c_1$. 

\paragraph{Functional spaces.} When $\mathcal X=L_2$, the convergence rate of $\mathbb E W_1(\hat\mu,\mu)$ is typically sub-polynomial in $n$, so the polynomial concentration rate given in \Cref{cor:concentration} when $p\in [1,2)$ can be quite meaningful.  To better understand the sub-exponential condition \eqref{eq:subexp} for functional data, we consider the Karhunen-Lo\`{e}ve decomposition \eqref{eq:KL-decomp}.  The sub-exponential condition can be satisfied under standard decay assumptions of the principal score variance $\sigma_m^2$, and corresponding tail bounds on the standardized principal component scores $Z_m$. The following proposition can be obtained after some straightforward applications of triangle inequality and some basic properties of Orlicz $\psi_{\alpha}$ norms.
\begin{proposition}\label{pro:function_orlicz}
  Assuming the K-L decomposition \eqref{eq:KL-decomp}, then
$\left\| \|X_i\| \right\|_{\psi_1} <\infty$ if one of the following holds.
\begin{enumerate}
  \item $\sum_{m=1}^\infty \sigma_m^2 <\infty$, and $\sup_{m}\|Z_m\|_{\psi_2}<\infty$.
  \item $\sum_{m=1}^\infty \sigma_m<\infty$, and $\sup_{m}\|Z_m\|_{\psi_1}<\infty$.
\end{enumerate}
\end{proposition}
Part 1 of \Cref{pro:function_orlicz} covers square-integrable Gaussian processes as an important special case (i.e., $Z_m\sim N(0,1)$ for all $m$).  Part 2 suggests that the tail condition on the PC scores $Z_m$ can be relaxed at the cost of a faster decay of eigenvalues.

\subsection{Concentration by Lipschitz property of $W_p(\hat\mu_n,\mu)$}
One can obtain concentration results for larger values of $p$ by exploiting the Lipschitz property of $W_p(\hat\mu_n,\mu)$ as a function of $(x_1,...,x_n)\in \mathcal X^n$. Similar ideas have been explored in the special case of $\mathcal X=\mathbb R$ (see Section 7.1 of \cite{BobkovL14}, and also \cite{GozlanC10}).

For $x=(x_1,...,x_n)$ and $x'=(x_1',...,x_n')$, define distance $\|x-x'\|^2=\sum_{i=1}^n \|x_i-x_i'\|^2$. Then one has, by triangle inequality,
\begin{align*}
  \left|W_p(\hat\mu,\mu)-W_p(\hat\mu',\mu)\right|\le W_p(\hat\mu,\hat\mu')\le n^{-1/p}\left(\sum_{i=1}^n\|x_i-x_i'\|^p\right)^{1/p}\le n^{-\frac{1}{2\vee p}}\|x-x'\|\,.
\end{align*}
Therefore, $W_p(\hat\mu,\mu)$ is $n^{-\frac{1}{2\vee p}}$-Lipschitz as a function from $\mathcal X^n$ to $\mathbb R$.  Such a Lipschitz property can lead to useful mean-concentration results provided the random vector $X$ has nice tail behavior.  There are many different characterizations of such tail behavior of $X$ \citep{BoucheronLM13}.  Here we use log Sobolev inequality for its presentation simplicity and generality.  Extension to other similar conditions, such as the Poincar\'{e} inequality, shall be straightforward.

\begin{definition}\label{def:LSI}
A probability measure $\mu$ on a metric space $\mathcal X$ is said to satisfy a log Sobolev inequality with constant $C$, if
$$
\mathbb E_\mu \left[f(X)^2 \log f(X)^2 \right]- \mathbb E_\mu\left[ f(X)^2\right] \log \mathbb E_\mu \left[f(X)^2\right]\le 2C \mathbb E_\mu \|\nabla f(X)\|^2 
$$
 for all smooth function $f:\mathcal X\mapsto \mathbb R$ such that the expectations are finite, where $\|\nabla f(x)\|=\sup_{y}\lim_{t\downarrow 0}\frac{|f(x+ty)-f(x)|}{t\|y\|}$.
\end{definition}
The log Sobolev inequality holds for strongly log-concave densities on $\mathcal X=\mathbb R^d$.  If the density of $X$ has the form $e^{-U(x)}$ with $U\succeq I/C$, then it satisfies \Cref{def:LSI} with the same constant $C$. For more details about log Sobolev inequalities, see \cite{Ledoux99}.

Combining the Lipschitz property of $W_p(\hat\mu,\mu)$ and the log Sobolev inequality under product measure \citep[Theorem 5.3 and Corollary 5.7]{Ledoux99}, we have
\begin{theorem}\label{thm:LSI-concentation}
  If the distribution $\mu$ satisfies log Sobolev inequality with constant $C$, then
  $$\mathbb P\left[W_p(\hat\mu,\mu)-\mathbb E W_p(\hat\mu,\mu)\ge t\right]
  \le \exp\left(-\frac{n^{\frac{2}{2\vee p}}t^2}{2C}\right)\,.$$
\end{theorem}

\Cref{thm:LSI-concentation} gives the same rate as \Cref{cor:concentration} for $p=1$, with strict improvement for $p > 1$ and non-trivial bounds for all values of $p$. In particular, with $\mathcal X=\mathbb R^d$ and a strongly log-concave distribution $\mu$, the mean concentration of $W_p(\hat\mu,\mu)$ has order $n^{-1/p}$, and is dominated by the mean value $\mathbb E W_p(\hat\mu,\mu)$, which, as implied by \Cref{thm:euclidean}, has order $n^{-\frac{1}{(2p)\vee d}}$.

\appendix
\section{Additional background for optimal transport}\label{app:additional}
\begin{definition}[Transport and coupling]
Let $\mu$, $\nu$ be two measures supported on the same measurable set $F\subseteq\mathcal X$ with the same amount of mass $\mu(F)=\nu(F)>0$.
A \emph{coupling} between $\mu$ and $\nu$ is a measure $\xi$ on $F\times F$ such that $\int_{y\in F}\xi(dx,dy)=\mu(dx)$ and $\int_{x\in F}\xi(dx,dy)=\nu(dy)$.

Equivalently, a coupling $\xi$ between $\mu$ and $\nu$ can also be represented by the corresponding transition kernel $K(\cdot,\cdot)$ such that for each $x\in F$, $K(x,\cdot)$ is a measure on $F$, and $\int_{x\in F}K(x,\cdot)\mu(dx)=\nu$.  We call such a transition kernel a \emph{transport} from $\mu$ to $\nu$.
\end{definition}

The following lemma is due to \cite{DereichSS13}, which is the building block for the construction of optimal transport from $\mu$ to $\hat\mu$.
\begin{lemma}\label{lem:DSS_trans}
  Let $\mu$, $\nu$ be two measures with same mass on their common support $F$. Let $C_1,...,C_k$ be a partition of $F$ such that $$\nu|_{C_i}=\frac{\nu(C_i)}{\mu(C_i)}\mu|_{C_i}\,,~~\forall~1\le i\le k\,.$$
Then there exists a coupling $\xi$ between $\mu$ and $\nu$ such that 
$$
\xi(\{(x,y)\in F^2:x\neq y\})=\frac{1}{2}\sum_{i=1}^k\left|\nu(C_i)-\mu(C_i)\right|\,.
$$
\end{lemma}
\begin{proof}
  Let $\delta=\frac{1}{2}\sum_{i=1}^k |\mu(C_i)-\nu(C_i)|$, and define
  \begin{align*}
    \xi=\xi_1+\delta^{-1}\xi_2
  \end{align*}
  where $\xi_1$ is the measure induced from the univariate measure $\mu\wedge \nu$ after the mapping $x\mapsto (x,x)$, and $\xi_2$ is the product measure
  $(\mu-\nu)_+\times (\nu-\mu)_+$. The claim can be directly verified.
\end{proof}

\section{Proofs for the general bound}\label{app:proof_gen}
In this section we provide proofs for \Cref{lem:bounded}, \Cref{lem:telescope}, and \Cref{thm:general}.  \Cref{lem:bounded} has been proved in \cite{WeedB17}, here we present an alternative proof based on the Markov chain method in \cite{DereichSS13}.  The proofs of \Cref{lem:telescope} and \Cref{thm:general} combines the proof ideas in \cite{FournierG15} with the general and finite partition approach.

\begin{proof}[Proof of \Cref{lem:bounded}]
Let $\mu_0=\mu$. For $\ell\ge 0$, define measure $\mu_{\ell+1}$ as
\begin{align*}
  \mu_{\ell+1} =\sum_{F\in\mathcal A_{\ell+1}}  \frac{\nu(F)}{\mu_\ell(F)}\mu_\ell\big|_F\,.
\end{align*}
Again we use the convention $0/0=0$.

\textbf{Facts.}\\
\textit{
1. For all $F\in\mathcal A_\ell$ we have $\mu_{\ell+1}(F)=\mu_\ell(F)=\nu(F)$.\\
2. $\mu_{\ell+1}\big|_C= \frac{\nu(C)}{\mu(C)}\mu\big|_C$, for  $C\in \mathcal A_{\ell+1}$.\\
3. $W_{p}^p(\mu_\ell, \nu)\le 2^p3^{-p\ell}$.
}

Facts 1-2 are straightforward to verify. To see the third one, use the fact that $\mu_\ell(F)=\nu(F)$ for all $F\in\mathcal A_\ell$, one can construct a transport from $\mu_\ell|_F$ to $\nu|_F$ using a product measure, with cost no more than $\nu(F)[{\rm diam}(F)]^p$.  The desired claim follows by summing over all $F\in\mathcal A_\ell$.

Next we need a low-cost transport from $\mu_\ell$ to $\mu_{\ell+1}$. 
Due to facts 1 and 2 above, for each $F\in\mathcal A_\ell$, \Cref{lem:DSS_trans} ensures
the existence of a coupling $\xi_{F}$ between $\mu_\ell|_F$ and $\mu_{\ell+1}|_F$ such that
\begin{align*}
\xi_F(\{(x,y)\in F^2:x\neq y\})\le & \frac{1}{2}\sum_{C\in\mathcal A_{\ell+1}, C\subseteq F}|\mu_{\ell+1}(C)-\mu_\ell(C)|\\
= &\frac{1}{2}\sum_{C\in\mathcal A_{\ell+1}, C\subseteq F}\left|\nu(C)-\frac{\nu(F)}{\mu(F)}\mu(C)\right|\,.
\end{align*}

Let $$\xi_\ell=\sum_{F\in\mathcal A_\ell}\xi_F\,.$$  We have
\begin{align*}
\xi_\ell\left(\left\{(x,y):x\neq y\right\}\right)=\frac{1}{2}\sum_{F\in \mathcal A_\ell} \sum_{C\in F\cap\mathcal A_{\ell+1}}\left|\nu(C)-\nu(F)\frac{\mu(C)}{\mu(F)}\right|
\end{align*}

 Let $Z_0,...,Z_{\ell^*}$ be a Markov chain with $Z_0\sim \mu$ and $(Z_{\ell+1}\mid Z_{\ell}=z)\sim K_\ell(z,\cdot)$, where $K_\ell$ is the transition kernel corresponding to $\xi_\ell$.  By construction, $K_\ell$ only moves points inside each $F\in \mathcal A_\ell$ so $\|Z_\ell-Z_{\ell'}\|\le 2\times 3^{-\ell}$ for all $\ell'>\ell$.
 
Let $L=\inf\{0\le\ell\le \ell^*-1:Z_{\ell+1}\neq Z_\ell\}$. If $Z_0=...=Z_{\ell^*}$ then define $L=\infty$. By construction we have
\begin{align*}
  Z_0=...=Z_L\,,~~~
  \|Z_L-Z_\ell\|\le 3^{-L},~~\forall~\ell>L\,.
\end{align*}

Then,
\begin{align*}
  &\mathbb E\|Z_0-Z_{\ell^*}\|^p \le  \mathbb E\left[\mathbf 1(L< \ell^*) \|Z_L-Z_{\ell^*}\|^p\right]\le 2^p\mathbb E 3^{-pL}\\
  =&2^p\sum_{\ell= 0}^{\ell^*-1}3^{-p\ell} \mathbb P(L=\ell)\le
  2^p\sum_{\ell= 0}^{\ell^*-1}3^{-p\ell} \mathbb P(Z_\ell\neq Z_{\ell+1})\\
  =&2^{p-1}\sum_{\ell= 0}^{\ell^*-1} 3^{-p\ell} \sum_{F\in\mathcal A_\ell} \sum_{C\in \mathcal A_{\ell+1}, C\subseteq F}\left|\nu(C)-\nu(F)\frac{\mu(C)}{\mu(F)}\right|\\
    =&2^{p-1}\sum_{\ell= 0}^{\ell^*-1} 3^{-p\ell} \sum_{F\in\mathcal A_\ell} \sum_{C\in \mathcal A_{\ell+1}, C\subseteq F}\left|\nu(C)-\mu(C)+\mu(F)\frac{\mu(C)}{\mu(F)}-\nu(F)\frac{\mu(C)}{\mu(F)}\right|\\
  \le&  2^{p-1}\sum_{\ell= 0}^{\ell^*-1} 3^{-p\ell} \sum_{F\in\mathcal A_\ell} \sum_{C\in \mathcal A_{\ell+1}, C\subseteq F}\left[\left|\nu(C)-\mu(C)\right|+\frac{\mu(C)}{\mu(F)}\left|\nu(F)-\mu(F)\right|\right]\\
  =&2^{p-1}\sum_{\ell= 0}^{\ell^*-1} 3^{-p\ell}\left[
  \sum_{C\in \mathcal A_{\ell+1}}|\nu(C)-\mu(C)|+\sum_{F\in\mathcal A_\ell}|\nu(F)-\mu(F)|
  \right]\\
  \le &(1+3^p)2^{p-1}\sum_{\ell= 1}^{\ell^*} 3^{-p\ell}
  \sum_{F\in \mathcal A_{\ell}}|\nu(F)-\mu(F)|\,.
\end{align*}

Let $(Z_{\ell^*},Z)$ be distributed as the optimal coupling between $\mu_{\ell^*}$ and $\nu$. Then combining the last inequality and Fact 3 above we obtain
\begin{align*}
  W_p^p(\mu,\nu)\le & \mathbb E\|Z_0-Z\|^p\le 2^{p-1}\left( \mathbb E\|Z_0-Z_{\ell^*}\|^p+\mathbb E\|Z_{\ell^*}-Z\|^p\right)\\
  \le & c_p\left[3^{-p\ell^*}+\sum_{\ell= 1}^{\ell^*} 3^{-p\ell}
  \sum_{F\in \mathcal A_{\ell}}|\nu(F)-\mu(F)|\right]\,. \qedhere
\end{align*}
\end{proof}

\begin{proof}[Proof of \Cref{lem:telescope}]
  The moment condition ensures that $\mu$ and $\nu$ are supported on $\bigcup_{j\ge 0}B_j$.  The remaining of the proof mainly mimics that of \cite{FournierG15}.
  %


  Let $\tilde\xi_j$ be the optimal coupling between $\tilde\mu_j$ and $\tilde\nu_j$, and $\xi_j$ the image of $\tilde \xi_j$ after the map $(x,y)\mapsto 2^j (x,y)$.

  Let $\eta=\frac{1}{2}\sum_{j\ge 0}|\mu(B_j)-\nu(B_j)|$ and define
  $$
  \xi = \sum_{j\ge 0}(\mu(B_j)\wedge \nu(B_j))\xi_j + \eta^{-1} \alpha\times \beta
  $$
  where
  \begin{align*}
    \alpha=\sum_{j\ge 0}(\mu(B_j)-\nu(B_j))_+ \frac{\mu|_{B_j}}{\mu(B_j)}\,,~~
    \beta=\sum_{j\ge 0}(\nu(B_j)-\mu(B_j))_+ \frac{\nu|_{B_j}}{\nu(B_j)}\,.
  \end{align*}
  Then $\xi$ is a coupling between $\mu$ and $\nu$.  It also holds that $\alpha(\mathcal X)+\beta(\mathcal X)=\eta$.

  Now
  \begin{align*}
    &\int \int \|x-y\|^p \eta^{-1}\alpha(dx)\beta(dy)\\
    \le & 2^{p-1}\int\int (\|x\|^p+\|y\|^p) \eta^{-1}\alpha(dx)\beta(dy)\\
    =&2^{p-1}\eta^{-1}\left[\int \|x\|^p \alpha(dx) \int \beta(dy) +
    \int \alpha(dx)\int \|y\|^p\beta(dy)\right]\\
    \le & 2^{p-1}\sum_{j\ge 0}2^{jp}\left[(\mu(B_j)-\nu(B_j))_+ + (\nu(B_j)-\mu(B_j))_+\right]\\
    = & 2^{p-1} \sum_{j\ge 0}2^{pj}|\mu(B_j)-\nu(B_j)|.
  \end{align*}
  Then
  \begin{align}
    W_p^p(\mu,\nu)\le & \int\int\|x-y\|^p\xi(dx,dy)\nonumber\\
    =&\sum_{j\ge 0}(\mu(B_j)\wedge\nu(B_j)) \int\int \|x-y\|^p\xi_j(dx,dy)+
    \int\int\|x-y\|^p \eta^{-1}\alpha(dx)\beta(dy)\nonumber\\
    \le&\sum_{j\ge0}2^{jp} (\mu(B_j)\wedge\nu(B_j)) W_p^p(\tilde\mu_j,\tilde\nu_j)+2^{p-1}\sum_{j\ge 0}2^{pj}|\mu(B_j)-\nu(B_j)|\nonumber\\
    \le&\sum_{j\ge 0}2^{pj}\left[(\mu(B_j)\wedge\nu(B_j)) W_p^p(\tilde\mu_j,\tilde\nu_j)+2^{p-1}|\mu(B_j)-\nu(B_j)|\right]\,,\label{eq:tele1}
  \end{align}  
which concludes the proof.
\end{proof}


\begin{proof}[Proof of \Cref{thm:general}]
  Without loss of generality we assume $\|\rho(X)\|_q=1$, since otherwise we can always consider $X/\|\rho(X)\|_q$.

With $\nu=\hat\mu$, the plan is to use \Cref{lem:bounded} to control the term $W_p^p(\tilde\mu_j,\tilde\nu_j)$ in \Cref{lem:telescope}.  The fact that $\nu=\hat\mu$ and the moment condition imply the conditions required by \Cref{lem:telescope}.
First observe that
  \begin{align*}
  & (\mu(B_j)\wedge\nu(B_j)) \sum_{\ell=1}^{\ell^*}3^{-p\ell}\sum_{F\in\mathcal A_\ell}|\tilde\mu(F)-\tilde\nu(F)|\\
  \le&\mu(B_j)\sum_{\ell=1}^{\ell^*}3^{-p\ell}\sum_{F\in\mathcal A_\ell}|\tilde\mu(F)-\tilde\nu(F)|\\
  \le & \mu(B_j)\sum_{\ell=1}^{\ell^*}3^{-p\ell}\sum_{F\in\mathcal A_\ell}\left|\frac{\mu((2^jF)\cap B_j)}{\mu(B_j)}-\frac{\nu((2^jF)\cap B_j)}{\nu(B_j)}\right|\\
  \le & \sum_{\ell=1}^{\ell^*}3^{-p\ell}\sum_{F\in\mathcal A_\ell}\left[|\mu((2^jF)\cap B_j)-\nu((2^j F)\cap B_j)|+\left|1-\frac{\mu(B_j)}{\nu(B_j)}\right|\nu((2^j F)\cap B_j)\right]\\
  \le &\sum_{\ell=1}^{\ell^*}3^{-p\ell}\sum_{F\in\mathcal A_\ell}|\mu((2^jF)\cap B_j)-\nu((2^j F)\cap B_j)| + \frac{3^{-p}}{1-3^{-p}}|\mu(B_j)-\nu(B_j)|\,.
  \end{align*}

Now applying \Cref{lem:bounded} to each of $W_p^p(\tilde\mu_j,\tilde\nu_j)$, combining with the last inequality, and plugging in \eqref{eq:tele1}, we get
  \begin{align}
    W_p^p(\mu,\nu)\le & c_p \sum_{j\ge 0}2^{pj}\Bigg[
    (\mu(B_j)\wedge\nu(B_j))3^{-p\ell^*}+|\mu(B_j)-\nu(B_j)|\nonumber\\
    &\quad+\sum_{\ell=1}^{\ell^*}3^{-p\ell}\sum_{F\in\mathcal A_\ell}|\mu((2^jF)\cap B_j)-\nu((2^j F)\cap B_j)|+|\mu(B_j)-\nu(B_j)|
    \Bigg]\nonumber\\
    \le& c_p\sum_{j\ge 0}2^{pj}\Bigg[\mu(B_j)3^{-p\ell^*}
    +\sum_{\ell=0}^{\ell^*}3^{-p\ell}\sum_{F\in\mathcal A_\ell}|\mu((2^jF)\cap B_j)-\nu((2^j F)\cap B_j)|\Bigg]\,.\label{eq:tele2}
  \end{align}

By the moment condition, Markov's inequality implies that
  $$
  \mu(B_j)\le 2^{-q(j-1)}\,.
  $$
 Because $n \hat\mu(A)$ is a binomial random variable with parameters $\mu(A)$ and $n$, we  have $\mathbb E|\hat\mu(A)-\mu(A)|\le (2\mu(A))\wedge\sqrt{\mu(A)/n}$ for all $A$. 
Therefore,
  \begin{align*}
  &\sum_{F\in\mathcal A_\ell} \mathbb E|\hat\mu((2^jF)\cap B_j)-\mu((2^jF)\cap B_j)|\\
  \le & \sum_{F\in\mathcal A_\ell} [2\mu((2^jF)\cap B_j)] \wedge \sqrt{\mu((2^jF)\cap B_j)/n}\\ \le & [2\mu(B_j)]  \wedge \left[ \sqrt{\frac{{\rm card}(\mathcal A_{\ell})\mu(B_j)}{n}}\right]\,,
  \end{align*}
  where the last step uses Cauchy-Schwartz.

Now the proof is complete if we can show that there exists a sequence of nested partitions $\mathcal A_1,...,\mathcal A_{\ell^*}$ such that $\sup_{F\in\mathcal A}{\rm diam}(F)\le 2\times 3^{-\ell}$, and ${\rm card}(\mathcal A_{\ell})\le \bar N_\ell=N_{3^{-(\ell+1)}}(B_0)$.  Proposition 3 of \cite{WeedB17} guarantees the existence of such a sequence of partitions.
\end{proof}

\section{Proofs for Euclidean and functional spaces}\label{app:special_cases}
\begin{proof}[Proof of \Cref{thm:euclidean}]
Recall that now $B_0$ is the unit ball in $\mathbb R^d$.  Use the result of \cite{Verger-Gaugry05}, we have, for some constant $c$
  $$
  \bar N_\ell = N_{3^{-(\ell+1)}}(B_0)\le 3^{(\ell+c)d}
  $$
Then \Cref{thm:general} implies
  \begin{align*}
   W_p^p(\hat\mu_n,\mu)
      \le & c_p\sum_{j\ge 0}2^{pj}\Bigg\{2^{-qj}3^{-p\ell^*}+\sum_{\ell=0}^{\ell^*}3^{-p\ell} \left[2^{-qj}  \wedge  \left(3^{d(\ell+c)/2}2^{-qj/2}n^{-1/2}\right)\right]\Bigg\}\,.
  \end{align*}
  Now let
  $$
  \ell_j^*=\left\lfloor\frac{\log_2 n-qj}{d\log_2 3}-c\right\rfloor\,,
  $$
  and
   $$\ell^*=\ell_0^*=\left\lfloor\frac{\log_2 n}{d\log_2 3}-c\right\rfloor\,.$$
If $\ell$ is an integer, then
  $$
  2^{-qj}  \ge  3^{d(\ell+c)/2}2^{-qj/2}n^{-1/2}\Leftrightarrow
  \ell \le \ell_j^*\,.
  $$

Let $j_n^*=\sup\{j\in\mathbb Z:\ell_j^*\ge 0\}=\lfloor q^{-1}(\log_2 n-cd\log_2 3) \rfloor$.

If $d\ge (2c)^{-1}\log_3 n$ then $n^{-(1/d)}\ge 3^{-2c}$ is a constant and the claim 
of theorem follows from \Cref{thm:general}, where the right hand side is trivially bounded by a constant depending only on $(p,q)$.

Now we focus on the case $d\ge (2c)^{-1}\log_3 n$, which implies that $j_n^*\ge 0$.

When $j\le j_n^*$, then $\ell_j^*\ge 0$, and (``$\lesssim$'' means up to a constant factor $c_{p,q}$)
  \begin{align*}
    &\sum_{\ell=0}^{\ell^*}3^{-p\ell} \left[2^{-qj}  \wedge  \left(3^{d(\ell+c)/2}2^{-qj/2}n^{-1/2}\right)\right]\\
   \le& \sum_{\ell=0}^{\ell_j^*}3^{-p\ell} 3^{d(\ell+c)/2}2^{-qj/2}n^{-1/2} +\sum_{\ell=\ell_j^*+1}^{\ell^*}3^{-p\ell} 2^{-qj}\\
   \lesssim & 2^{-qj/2}n^{-1/2}3^{cd/2}\sum_{\ell=0}^{\ell_j^*}3^{-(p-d/2)\ell}+
   2^{-qj}(2^{-qj}n)^{-p/d}\,.
  \end{align*}
  When $j> j_n^*$, then $\ell_j^*<0$, and
  \begin{align*}
    \sum_{\ell=0}^{\ell^*}3^{-p\ell} \left[2^{-qj}  \wedge  \left(3^{d(\ell+c)/2}2^{-qj/2}n^{-1/2}\right)\right] \le \sum_{\ell=0}^{\ell^*}3^{-p\ell} 2^{-qj}\lesssim 2^{-qj}\,.
  \end{align*}
  Thus $\sum_{j>j_n^*}2^{pj}2^{-qj}\lesssim 2^{-qj_n^*(1-p/q)}\lesssim (n3^{-cd})^{-(1-p/q)}$ and
  \begin{align}
  &\sum_{j\ge 0} 2^{pj}  \sum_{\ell=0}^{\ell^*}3^{-p\ell} \left[2^{-qj}  \wedge  \left(3^{d(\ell+c)/2}2^{-qj/2}n^{-1/2}\right)\right]\nonumber\\
  \lesssim & \sum_{j=0}^{j_n^*}2^{pj}\Bigg\{2^{-qj/2}n^{-1/2}3^{cd/2}\sum_{\ell=0}^{\ell_j^*}3^{-(p-d/2)\ell}+
   2^{-qj}(2^{-qj}n)^{-p/d}\Bigg\}+(n3^{-cd})^{-(1-p/q)}\,.\label{eq:euclid0}
  \end{align}

  Case 1: $p>d/2$.  The terms in the sum $\sum_{\ell=0}^{\ell_j^*}3^{-(p-d/2)\ell}$ is geometrically decreasing and hence the sum is bounded by the first term by a constant factor. Moreover, for $j\le j_n^*$ we have $(2^{-qj}n)^{-p/d} \le c_{p,q}(2^{-qj}n)^{-1/2}$. So \eqref{eq:euclid0} is bounded by (ignoring constant factors)
  \begin{equation}
   \sum_{j=0}^{j_n^*} 2^{pj}2^{-qj/2}n^{-1/2}+n^{-(1-p/q)}\,.\label{eq:euclid1}
  \end{equation}
  Case 1.1: $p< q/2$. Both terms in \eqref{eq:euclid1} are bounded by $n^{-1/2}$.

  Case 1.2: $p=q/2$.  The first term in \eqref{eq:euclid1} equals $(1+j_n^*) n^{-1/2}\le c_{p,q} (\log n) n^{-1/2}$, and the second term in \eqref{eq:euclid1} equals $n^{-1/2}$.

  Case 1.3: $p\in (q/2,q)$.  The first sum in \eqref{eq:euclid1} equals $n^{-1/2}2^{(p-q/2)j_n^*}\le c_{p,q}n^{-(1-p/q)}$, the same as the second term.

  Case 2: $p=d/2$.  Now \eqref{eq:euclid0} reduces to (ignoring constant factors)
  \begin{equation}\label{eq:euclid2}
 \sum_{j=0}^{j_n^*}2^{pj} 2^{-qj/2}n^{-1/2}(\ell_j^*+1)+n^{-(1-p/q)}\,.
  \end{equation}
  Case 2.1: $p<q/2$.   The first term is a geometric sum bounded by $n^{-1/2}\log n$, which dominates the second term.

  Case 2.2: $p=q/2$.  The first term is of order $n^{-1/2}(\log n)^2$, which dominates the second term.

  Case 2.3: $p>q/2$.  
  The first term becomes
  $n^{-1/2}  \sum_{j=0}^{j_n^*} 2^{(p-q/2)j} (\ell_j^*+1)$.
Let $\tilde j_n^*=q^{-1}(\log_2 n-cd\log_2 3)$, and $a=2^{p-q/2}$
\begin{align}
  \sum_{j=0}^{j_n^*} 2^{(p-q/2)j} (\ell_j^*+1) \le & 
 \sum_{j=0}^{j_n^*} 2^{(p-q/2)j} \left(\frac{\log_2 n -qj}{d\log_2 3}-c+1\right)\nonumber\\
 =&\frac{q}{d\log_2 3}\sum_{j=0}^{j_n^*} a^j(\tilde j_n^*-j+c_{p,q})\nonumber\\
 \le & c_{p,q}\sum_{j=0}^{j_n^*} a^j(j_n^*-j)+ c_{p,q}a^{j_n^*}\,.\label{eq:euclid2.3}
\end{align}
Standard calculation shows that 
$$\sum_{j=0}^{j_n^*} a^j(j_n^*-j)=\frac{a}{(a-1)^2}(a^{j_n^*}-1)-\frac{j_n^*}{a-1}\le c_{p,q}a^{j_n^*}\,.$$
Thus \eqref{eq:euclid2.3} is bounded by
$$
c_{p,q}2^{(p-q/2)j_n^*}\le c_{p,q} n^{p/q-1/2}
$$
and \eqref{eq:euclid2} is bounded by $c_{p,q}n^{-(1-p/q)}$\,.

Case 3: $p<d/2$. In this case $3^{-(p-d/2)\ell}$ is an increasing geometric sequence as $\ell$ changes from $0$ to $\ell_j^*$. Thus \eqref{eq:euclid0} becomes, using the fact that $3^{\ell_j^*}\le 3^{-c}(n2^{-qj})^{1/d}$,
\begin{align}
&\sum_{j=0}^{j_n^*}2^{pj}\left\{2^{-qj/2}n^{-1/2}3^{cd/2}3^{(d/2-p)\ell_j^*}+2^{-qj}(2^{-qj}n)^{-p/d}\right\}+(n3^{-cd})^{-(1-p/q)}\nonumber\\
\le&\sum_{j=0}^{j_n^*}2^{-jq(1-p/d-p/q)}n^{-p/d}+(n3^{-cd})^{-(1-p/q)}\,.\label{eq:euclid3}
\end{align}
  Case 3.1: $d< \frac{qp}{q-p}$ (or $1-p/d-p/q<0$). The first term is the sum of an increasing geometric sequence, which is bounded by the last term $n^{-p/d}2^{j_n^*q(p/d+p/q-1)}\le c_{p,q}n^{-(1-p/q)}$, matching the second term.

  Case 3.2: $d= \frac{qp}{q-p}$ (or $1-p/d-p/q=0$).  The first term is of order $n^{-p/d}\log n$ and dominates the second term.

  Case 3.3: $d> \frac{qp}{q-p}$ (or $1-p/d-p/q>0$). The first term is of order $n^{-p/d}$.  Now if $d\le 2qp/(q-p)$ then the second term is bounded by $n^{-(1-p/q)}$ and hence dominated by the first term.  If $d> 2qp/(q-p)$ then by assumption that $d\le (2c)^{-1}\log_3 n$ we have $3^{-cd}\ge n^{-1/2}$ so the second term is bounded by $n^{-(1-p/q)/2}$ which is dominated by the first term.

Finally, note that the contribution from $\sum_{j\ge 0}2^{-pj}2^{-qj}3^{-p\ell^*}$ is bounded by $n^{-p/d}$, which is dominated in all the nine cases.
\end{proof}

\begin{proof}[Proof of \Cref{thm:poly}]
Again it suffices to prove for the case $M_q=1$.  In our proof ``$\lesssim$'' means inequality holds up to a constant depending only on $(p,q,b)$.
  
According to Corollary 2.4 of \cite{LuschgyP04}, we have, as $\epsilon\rightarrow 0$,
  $$
  \log N_\epsilon(B_0)=(1+o(1)) b \epsilon^{-\frac{1}{b}}\,.
  $$
As a consequence we have, for some positive constant $c_b$ fixed through this proof, and all $\ell\in \mathbb N$,
$$
\bar N_\ell\le 2^{c_b 3^{\ell/b}}\,.
$$

Again, in order to use \Cref{thm:general}, the key is to break down the term
$$
2^{-qj}  \wedge  \left(\bar N_\ell^{1/2}2^{-qj/2}n^{-1/2}\right)\,.
$$

  Let $$\ell_j^*=\left\lfloor b\left[\log_3(\log_2 n-qj)-\log_3 c_b\right]\right\rfloor$$
  and $\ell^*=\ell_0^*=\left\lfloor b(\log_3 \log_2 n-\log_3 c_b)\right\rfloor$ so that
  $\ell\le \ell_j^*\Rightarrow 2^{-qj}\ge \bar N_\ell^{1/2}2^{-qj/2}n^{-1/2}$.

  Let $j_n^*=\lfloor q^{-1}(\log_2 n-c_b)\rfloor$ so that $j\le j_n\Leftrightarrow \ell_j^*\ge 0$.

If $\log_2 n< c_b$ then $n< 2^{c_b}=O(1)$.  The claim of the theorem follows trivially from \Cref{thm:general}. Now we focus on the case $\log_2 n \ge c_b$ so that $j_n^*\ge 0$.

  So when $j\le j_n^*$,
  \begin{align}
    &\sum_{\ell=0}^{\ell^*}3^{-p\ell} \left[2^{-qj}  \wedge  \left(\bar N_\ell^{1/2}2^{-qj/2}n^{-1/2}\right)\right]\nonumber\\
  =&2^{-qj/2}n^{-1/2}\sum_{\ell=0}^{\ell_j^*} 3^{-p\ell} 2^{( c_b/2)3^{\ell/b}}  + 2^{-qj}\sum_{\ell=\ell_j^*+1}^{\ell^*}3^{-p\ell}\nonumber\\
  \lesssim & 2^{-qj/2}n^{-1/2}+ 2^{-qj}3^{-p\ell_j^*}\nonumber\\
  \lesssim & 2^{-qj/2}n^{-1/2}+2^{-qj}(\log_2 n - qj)^{-bp}\,.\nonumber
  \end{align}
To see the first inequality, observe that the terms $3^{-p\ell}2^{(c_b/2)3^{\ell/b}}$ becomes super-geometrically increasing with rate at least $2$ after a constant number of terms. Thus the sum is bounded by a constant plus the last term, which is $3^{-p\ell_j^*}2^{(c_b/2)3^{\ell_j^*/b}}\le 3^{-p\ell_j^*}n^{1/2}2^{-qj/2}$.

When $j> j_n^*$, we have $\ell_j^*< 0$ and hence
  \begin{align*}
    \sum_{\ell=0}^{\ell^*}3^{-p\ell} \left[2^{-qj}  \wedge  \left(\bar N_\ell^{1/2}2^{-qj/2}n^{-1/2}\right)\right]
  \lesssim & 2^{-qj}\,.
  \end{align*}

Combining the above two inequalities with \Cref{thm:general} we obtain
  \begin{align}
  &\mathbb E W_p^p(\hat\mu,\mu)\nonumber\\
      \lesssim & 3^{-p\ell^*}+\sum_{j=0}^{j_n^*}2^{pj}\left[2^{-qj/2}n^{-1/2}+ 2^{-qj}(\log_2 n - qj)^{-bp}\right] +\sum_{j>j_n^*}2^{pj}2^{-qj}\nonumber\\
      \lesssim & \sum_{j=0}^{j_n^*} 2^{-(q-p)j}(\log_2 n-qj)^{-bp}+
      n^{-1/2}\sum_{j=0}^{j_n^*} 2^{-(q/2-p)j}+(\log n)^{-bp}+n^{-(1-p/q)}\,,\label{eq:poly1}
  \end{align}
  where the $(\log n)^{-bp}$ upper bounds $3^{-p\ell^*}$, and $n^{-(1-p/q)}$ upper bounds $\sum_{j>j_n^*}2^{pj}2^{-qj}$.
  
Now we control the first two terms in \eqref{eq:poly1}.
For the first term, if we pick $c_b$ large enough so that $c_b\ge \frac{q2^{(q/p-1)/(2b)}}{2^{(q/p-1)/(2b)}-1}$, then for all $0\le j\le j_n^*$ the terms $2^{-(q-p)j}(\log_2 n -qj)^{-bp}$
in the first sum decreases super-geometrically with rate $2^{-(q-p)/2}$ and hence bounded by the first term, which is of order $(\log n)^{-pb}$. 

For the second term, we consider three cases.

Case 1: $p<q/2$, then the sequence $2^{-(q/2-p)j}$ is geometrically decreasing and the sum is bounded by the first term which is $1$. So the second part in \eqref{eq:poly1} is bounded by $n^{-1/2}$.

Case 2: $p=q/2$, the sum becomes $n^{-1/2}(1+j_n^*)\lesssim n^{-1/2}\log n$.

Case 3: $p>q/2$, the sum is dominated by the last term, which is
$n^{-1/2}2^{(p-q/2)j_n^*}\lesssim n^{-1/2}n^{p/q-1/2}=n^{-(1-p/q)}$.

In all these cases, \eqref{eq:poly1} is dominated by the first term and hence
  \begin{align*}
    \mathbb E W_p^p(\hat\mu,\mu)\lesssim (\log n)^{-bp}\,,
  \end{align*}
which implies the desired upper bound using $p\ge 1$ and Jensen's inequality.

Now we prove the lower bound.
Let $\epsilon_n=(\frac{2}{b}\log n)^{-b}$. By Corollary 2.4 of \cite{LuschgyP04} we have
$$
\lim_{n\rightarrow\infty}\frac{\log N_{\epsilon_n}(B_0)}{2\log n}= 1\,.
$$
Therefore, there exists a constant $n_0\in\mathbb N$ such that $N_{\epsilon_n}\ge n$ for all $n\ge n_0$.

For $n< n_0$, the result is trivial because the desired lower bound is essentially a constant and one can pick $\mu_n$ to be a distribution with two point masses that are a constant distance apart.

For $n\ge n_0$, let $x_1,...,x_n$ be points in $B_0$ such that $\min_{1\le i<j\le n}\|x_i-x_j\|\ge \epsilon_n$.  Let $\mu_n$ be the measure that puts $n^{-1}$ mass at each $x_i$ ($1\le i\le n$).

If $\hat\mu_n(\{x_i\})=0$ for some $x_i$, then any transport between $\hat\mu_n$ and $\mu_n$ costs probability mass at least $n^{-1}$ and distance at least $\epsilon_n$ to just cover the point $x_i$ alone.  Thus we have the following lower bound
$$
W_p(\hat\mu_n,\mu_n)\ge \left\{\sum_{i=1}^n \epsilon_n^p n^{-1}\mathbf 1\left[\hat\mu_n(\{x_i=0\})\right]\right\}^{1/p}=\epsilon_n \kappa_n^{1/p}\,,
$$
where $\kappa_n$ is the proportion of $x_i$'s such that $\hat\mu_n(\{x_i\})=0$.
As a result
  \begin{align*}
    \mathbb E W_p(\hat\mu_n,\mu_n)\ge \epsilon_n \mathbb E(\kappa_n^{1/p})\,.
  \end{align*}
Elementary calculation shows that $\mathbb E \kappa_n\rightarrow e^{-1}$. Also $\kappa_n$ has bounded difference in the sense that if one out of the $n$ independent sample points is changed arbitrarily, then $\kappa_n$ changes no more than $1/n$. Thus, by McDiarmid's inequality, $\mathbb P(|\kappa_n-\mathbb E\kappa_n|\ge t)\le 2\exp\left(-2nt^2\right)$ for all $t>0$. Now for $n$ larger than some constant $n_1$ we have $\mathbb E\kappa_n\ge (2e)^{-1}$.  Choosing $t=(4e)^{-1}$ we have $\mathbb P(\kappa_n\ge (4e)^{-1})\rightarrow 1$.  Therefore
$\mathbb E(\kappa_n^{1/p})\ge c$ for some universal positive constant $c$.
\end{proof}

\begin{proof}[Proof of \Cref{lem:metric_entropy_exp}]
  For $\epsilon\in(0,1)$, let $J_{1}=\lceil \log_\gamma(\epsilon^{-1}) \rceil$, so that $\tau_m/\epsilon >1 \Leftrightarrow m\le J_{1}$
  Let $B_{1}=\{x\in\mathbb R^{J_{1}}:\sum_{m=1}^{J_1}(x_m/\tau_m)^2\le 1\}$. Then $B_{1}$ is the projection of $B_0$ onto the first $J_{1}$ coordinates.  Thus $N_\epsilon(B_0)\ge N_{\epsilon}(B_{1})$.  But Theorem 1 of \cite{Dumer06} implies that
  \begin{align*}
   \log N_{\epsilon}(B_{1})\ge & \sum_{m=1}^{J_{1}} \log(\tau_m/\epsilon)= J_{1}\log(\epsilon^{-1})-\log\gamma \frac{J_{1}(J_{1}-1)}{2}
  \end{align*}
  and the lower bound claim follows by $\log_\gamma(\epsilon^{-1})\le J_{1}<\log_\gamma(\epsilon^{-1})+1$.

For the upper bound, 
let $\theta\in(0,1/2)$ be a constant. For example, we can pick $\theta=1/3$.
Define $J_{2}=\lceil \log_\gamma (\epsilon\sqrt{1-\theta})^{-1} \rceil$, so that $\tau_m/\epsilon\in (\sqrt{1-\theta},1]\Leftrightarrow m\in(J_{1},J_{2}]$.
Let \begin{align*}
  B_{2}=&\{x\in\mathbb R^{J_{2}}:\sum_{m=1}^{J_2}(x_m/\tau_m)^2\le 1\}\,,\\
  B_3=&\{x\in\mathbb R^\infty:\sum_{m\ge 1}(x_m/\tau_{J_2+m})^2\le 1\}\,.
\end{align*}
Then $B_0\subset B_2\times B_3$.

By construction $\tau_{J_2+1}/\epsilon \le \sqrt{1-\theta}$, so $B_3$ can be covered by the ball centered at $0$ with radius $\sqrt{1-\theta}\epsilon$.  As a consequence,
$$
N_{\sqrt{2-\theta}\epsilon}(B_0) \le N_{\sqrt{2-\theta}\epsilon}(B_2\times B_3)
\le N_\epsilon(B_2)\,.
$$
By Theorem 2 of \cite{Dumer06}, we have ($c_0,c_1$, depending only on $\gamma$, $\theta$, may change from line to line)
\begin{align*}
  \log N_\epsilon(B_2)\le& \sum_{m=1}^{J_1}\log(\tau_m/\epsilon)+J_2\log(3/\theta)\\
  =& \frac{\log\gamma}{2}\lceil\log_\gamma(\epsilon^{-1})\rceil\left(2\log_\gamma(\epsilon^{-1})-\lceil\log_\gamma(\epsilon^{-1})\rceil+1\right)\\
&\quad+\left\lceil\log_\gamma(\epsilon^{-1})+\log_\gamma\frac{1}{\sqrt{1-\theta}}\right\rceil\log(3/\theta)\\
\le & \frac{\log\gamma}{2}(\log_\gamma(\epsilon^{-1})+1)^2+c_1\log_\gamma(\epsilon^{-1})+c_0\\
= & \frac{1}{2\log\gamma}(\log(\epsilon^{-1})+c_1)^2+c_0\,.
\end{align*}
The desired result can be obtained by plugging in $\epsilon=3^{-(\ell+1)}/\sqrt{2-\theta}$.
\end{proof}

\begin{proof}[Proof of \Cref{thm:exp}]
  By \Cref{lem:metric_entropy_exp} we have for $\ell\in\mathbb N$
  $$
  \log \bar N_\ell \le c_\gamma (\ell+c_1)^2 +\log c_0
  $$
  for $c_\gamma=\frac{(\log 3)^2}{2\log\gamma}$ and positive constants $c_1$, $c_0$ depending on $\gamma$ only.
  Then $$\bar N_\ell\le c_0 e^{c_\gamma(\ell+c_1)^2}\,.$$

Let $j_n^*=\left\lfloor q^{-1}\frac{\log(n/c_0)-c_1^2c_\gamma}{\log 2}\right\rfloor$.

If $j_n^*<0$ then $n< c_0e^{c_1^2c_\gamma}$ and the claim follows trivially.
Now we focus on the case $j_n^*\ge 0$.

Let $\ell_j^*=\left\lfloor\left[\frac{(\log(n/c_0)-qj\log 2)_+}{c_\gamma}\right]^{1/2}-c_1\right\rfloor$,
  and $\ell^*=\ell_0^*=\left\lfloor\left[\frac{\log(n/c_0)}{c_\gamma}\right]^{1/2}-c_1\right\rfloor$.

  When $j\le j_n^*$, we have $\ell_j^*\ge 0$ so
  \begin{align*}
    &\sum_{\ell=0}^{\ell^*}3^{-p\ell} \left[2^{-qj}  \wedge  \left( \bar N_\ell^{1/2}2^{-qj/2}n^{-1/2}\right)\right]\\
  \lesssim&c_0^{1/2}2^{-qj/2}n^{-1/2}\sum_{\ell=0}^{\ell_j^*} 3^{-p\ell} e^{c_\gamma (\ell+c_1)^2/2}  + 2^{-qj}\sum_{\ell=\ell_j^*+1}^{\ell^*}3^{-p\ell}\\
  \lesssim&2^{-qj/2}n^{-1/2}+ 2^{-qj}3^{-p\sqrt{\frac{\log(n/c_0)-qj\log 2}{c_\gamma}}}\,.
  \end{align*}
To see the last inequality, observe that after a constant number of terms in the first sum in the line above, the terms $3^{-p\ell} e^{c_\gamma (\ell+c_1)^2/2}$ become super-geometrically increasing with rate at least $2$. So the first sum is upper bounded (up to constant factor) by the last term plus $1$.

When $j> j_n^*$, we have $\ell_j^*=0$
  and
  \begin{align*}
    \sum_{\ell=0}^{\ell^*}3^{-p\ell} \left[2^{-qj}  \wedge  \left( \bar N_\ell^{1/2}2^{-qj/2}n^{-1/2}\right)\right]
  \lesssim & 2^{-qj}\,.
  \end{align*}

Combining the above two inequalities with \Cref{thm:general}, we get
  \begin{align}
  &\mathbb E W_p^p(\hat\mu,\mu)\nonumber\\
      \lesssim & 3^{-p\ell^*}+\sum_{j=0}^{j_n^*}2^{pj}2^{-qj/2}n^{-1/2}+\sum_{j=0}^{j_n^*}2^{pj} 2^{-qj}3^{-p\sqrt{\frac{\log(n/c_0) -qj\log 2}{c_\gamma}}} +\sum_{j>j_n^*}2^{pj}2^{-qj}\nonumber\\
      \lesssim & \sum_{j=0}^{j_n^*}2^{pj} 2^{-qj}3^{-p\sqrt{\frac{\log(n/c_0) -qj\log 2}{c_\gamma}}}+\sum_{j=0}^{j_n^*}2^{pj}2^{-qj/2}n^{-1/2}+3^{-p\sqrt{\log n/c_\gamma}}+n^{-(1-p/q)}\,,\label{eq:exp0}
  \end{align}
where the term $3^{-p\sqrt{\log n/c_\gamma}}$ controls $3^{-p\ell^*}$, and $3^{-p\sqrt{\log n/c_\gamma}}+n^{-(1-p/q)}$ controls $\sum_{j>j_n^*}2^{pj}2^{-qj}$.

  If we choose $c_1$ large enough so that $c_1\ge\frac{qp}{q-p}\frac{\log 3}{c_\gamma}$ then it can be verified that the terms in the first sum in \eqref{eq:exp0} is  super-geometrically decreasing with rate $2^{-(q-p)}$. Thus the sum is bounded by the first term
  $3^{-p\sqrt{(\log (n/c_0))/c_\gamma}}\lesssim 3^{-p\sqrt{(\log n)/c_\gamma}}$.

The second sum in \eqref{eq:exp0} can be controlled using the same argument as for the second term in \eqref{eq:poly1} in the proof of \Cref{thm:poly}, which is bounded by $n^{-[(1-p/q)\wedge(1/2)]}(\log n)^{\mathbf 1(q=2p)}$, and is dominated by the first term in \eqref{eq:exp0}.

  Therefore the first and third terms dominate in \eqref{eq:exp0} and final rate is
  $$
  \mathbb E W_p^p(\hat\mu,\hat\mu)\lesssim 3^{-p\sqrt{(\log n)/c_\gamma}}=
  e^{-p\sqrt{2\log\lambda\log n}}\,,
  $$
which concludes the proof of upper bound.

For the lower bound.
  Let $\epsilon_n=e^{-\sqrt{2\log\gamma\log n}}$ then by \Cref{lem:metric_entropy_exp} we have
  $$
  \log N_\epsilon \ge \frac{1}{2\log\gamma}\left[\log(1/\epsilon)\right]^2=\log n.
  $$
  So $N_\epsilon \ge n$.
  Let $x_1,...,x_n\in B_0$ be such that $\min_{i\neq j}\|x_i-x_j\|\ge \epsilon$.
  Let $\mu_n$ be the distribution putting $n^{-1}$ mass at each $x_i$. The rest of the proof are identical to the lower bound proof of \Cref{thm:poly}. \end{proof}

\begin{proof}[Proof of \Cref{pro:fpc}]
  We only prove for the polynomial decay case, the exponential case is similar. Without loss of generality we assume $c_0=M=1$, as both constants can be recovered after scaling. Then
  \begin{align*}
\|\rho_{\tau}(X)\|_q^q =  &  \mathbb E\left\{\left[\sum_{m\ge 1}\left(\frac{X_m}{\tau_m}\right)^2\right]^{q/2}\right\}
  = \left\|\sum_{m\ge 1}\left(\frac{X_m}{\tau_m}\right)^2\right\|_{q/2}^{q/2}\\
  \le &\left[\sum_{m\ge 1}\left\|\left(\frac{X_m}{\tau_m}\right)^2\right\|_{q/2}\right]^{q/2}
  =  \left[\sum_{m\ge 1} \left\|\frac{X_m}{\tau_m}\right\|_q^2\right]^{q/2}\\
  = & \left[\sum_{m\ge 1} \frac{\|X_m\|_q^2}{\tau_m^2}\right]^{q/2}
  \le  \left[\sum_{m\ge 1}(\sigma_m/\tau_m)^2\right]^{q/2}\,,
  \end{align*}
  where the inequality uses Minkowski inequality since we assume $q/2\ge 1$. The claimed results follow by replacing $\sigma_m$ by the assumed upper bound, and $\tau_m$ by the assumed form in \eqref{eq:poly_tau} or \eqref{eq:exp_tau}.
\end{proof}

\section{Proof of concentration inequality}\label{app:concentration}
\begin{proof}[Proof of \Cref{thm:Bernstein-McDiarmid}]
Let $\Delta_i=\mathbb E\left[f(X_1,...,X_n)\mid X_{1}^i\right]-\mathbb E\left[f(X_1,...,X_n)\mid X_{1}^{i-1}\right]$, where $X_i^j$ denotes $(X_i,X_{i+1},...,X_j)$.

For two random vectors $(Y,Z)$ and a function $g$, we use the notation $\mathbb E_{Y}g(Y,Z)$ to denote the conditional expectation of $g(Y,Z)$ given $Z$ (the notation means that we integrate over $Y$ in $g(Y,Z)$).

  For any $i$ and $c>0$
  \begin{align*}
    &\mathbb E(e^{c\Delta_i}\mid X_1^{i-1})\\
    =&\mathbb E_{X_i} \exp \left\{c\left[\mathbb E_{X_{i+1}^n} f(X) - \mathbb E_{X_{i+1}^n, X_i'} f(X_{(i)}') \right]\right\}\\
    =&\mathbb E_{X_i} \exp \left\{\mathbb E_{X_{i+1}^n, X_i'}c\left[ f(X) -  f(X_{(i)}') \right]\right\}\\
    \le & \mathbb E_{X_i} \mathbb E_{X_{i+1}^n,X_i'} \exp(cD_i)\\
    = & \mathbb E_{X_{i+1}^n}\mathbb E_{X_i,X_i'}\exp (cD_i)\,.
  \end{align*}

  Using the fact that $\mathbb E_{X_i,X_i'} D_i=0$, we get
  \begin{align*}
    \mathbb E_{X_i,X_i'}\exp (cD_i) = & \mathbb E_{X_i,X_i'}\left(1+cD_i+\sum_{k\ge 2}\frac{c^k D_i^k}{k!}\right)\\
    =&1+\mathbb E_{X_i,X_i'}\sum_{k\ge 2}\frac{c^k D_i^k}{k!}\\
    \le & 1 + \frac{1}{2}\sigma_i^2 c^2 \sum_{k\ge 0}(cM)^k\le \exp\left( \frac{1}{2}\sigma_i^2 c^2 \sum_{k\ge 0}(cM)^k\right)\\
    =&\exp\left( \frac{1}{2}\sigma_i^2 c^2 (1-cM)^{-1}\right)
    \,,
  \end{align*}
  provided that $cM<1$.

  For any $t,c>0$, $c< M^{-1}$, we have
  \begin{align*}
    &\mathbb P(f-\mathbb E f \ge t\sigma)
    \le  e^{-ct\sigma} \mathbb E e^{c(f-\mathbb E f)}
    =  e^{-ct\sigma} \mathbb E e^{c\sum_{i=1}^n \Delta_i}\\
    \le & e^{-ct\sigma} \mathbb E\left[e^{c\sum_{i=1}^{n-1}\Delta_i}\mathbb E(e^{c\Delta_n}\mid X_1^{n-1})\right]\\
    \le & e^{-ct\sigma} \exp\left( \frac{1}{2}\sigma_n^2 c^2 (1-cM)^{-1}\right)\mathbb E\left[e^{c\sum_{i=1}^{n-1}\Delta_i}\right]\\
    \le & e^{-ct\sigma} \prod_{i=1}^n\exp\left( \frac{1}{2}\sigma_i^2 c^2 (1-cM)^{-1}\right)
    =  e^{-ct\sigma}\exp\left(\frac{1}{2}\sigma^2 c^2 (1-cM)^{-1}\right)\,,
  \end{align*}
  where $\sigma^2=\sum_{i=1}^n\sigma_i^2$.

The claimed result follows by choosing $c=\frac{t}{\sigma+tM}$.
\end{proof}

%


\bibliographystyle{plain}
\bibliography{wasserstein}
\end{document}